\documentclass[10pt]{article}
\RequirePackage{hyperref}
\advance\textwidth by 3.2cm

\advance\textheight by 3.8cm

\usepackage{amsmath}
\usepackage{amsfonts}
\usepackage{amsthm}
\usepackage{layout}

\usepackage{dsfont}
\usepackage{enumerate}

\theoremstyle{plain}
\numberwithin{equation}{section}

\newtheorem{thm}{Theorem}[section]
\newtheorem*{thm*}{Theorem}
 
 \newtheorem{lem}[thm]{Lemma}
 \newtheorem{prop}[thm]{Proposition}
  \newtheorem{Assumption}{Assumption}
 \theoremstyle{definition}
 \newtheorem{defn}[thm]{Definition}
 \theoremstyle{remark}
 \newtheorem{rem}[thm]{Remark}
 \newtheorem{ex}[thm]{Example}

 \numberwithin{equation}{section}

\newcommand\new[1]{}
\newcommand\norm[1]{\left|\!\left| #1\right|\!\right|}

\newcommand{\embed}{\hookrightarrow}
\newcommand\BIP{\operatorname{BIP}}

\newcommand\Tr{\operatorname{Tr}}

\newcommand\range{\operatorname{Range}}

\newcommand\Vertt{|\!|}

\newcommand{\mC}{\mathcal{C}}

\newcommand{\N}{\mathcal{N}}

\def\R{{\mathds R}}

\def\Re{\mbox{Re\,}}

\newcommand{\Exp}{\mathds{E}}

\newcommand{\T}{\mathcal{T}}
\newcommand{\Prob}{\mathds{P}}

\newcommand{\Lin}{\mathcal{L}}
\newcommand{\abs}[1]{\left\vert#1\right\vert}
\newcommand{\set}[1]{\left\{#1\right\}}
\newcommand{\seq}[1]{\left<#1\right>}

\newcommand{\bE}{\mathbf{E}}
\newcommand{\bH}{\mathbf{H}}
\newcommand{\cH}{\mathcal{H}}
\newcommand{\B}{\mathcal{B}}
\newcommand{\A}{\mathcal{A}}

\newcommand{\Fil}{\mathcal{F}}

\usepackage[auth-lg]{authblk}

\title{\Large Backward Ornstein-Uhlenbeck transition operators and mild solutions of non-autonomous Hamilton-Jacobi equations in Banach spaces}

\author{Rafael Serrano \thanks{rafael.serrano@urosario.edu.co}}

\affil{\smallskip\small \textsc{Universidad del Rosario}\\
Calle 12C No. 4-69\\
Bogot\'a, Colombia}

\begin{document}
\bibliographystyle{alpha}

%\layout
%\email{rafael.serrano@urosario.edu.co}

%\subjclass{Primary 35R15; Secondary 60H30,37L55}

\date{\today}

\maketitle

\begin{abstract}
\noindent In this paper we revisit the mild-solution approach to second-order semi-linear PDEs of Hamilton-Jacobi type in infinite-dimensional spaces. We show that a well-known result on existence of mild solutions in Hilbert spaces can be easily extended to non-autonomous Hamilton-Jacobi equations in Banach spaces. The main tool is the regularizing property of Ornstein-Uhlenbeck transition evolution operators for stochastic Cauchy problems in Banach spaces with time-dependent coefficients.
%This includes the dynamic programming equation associated with certain optimal control problems in Banach spaces for non-autonomous %semi-linear stochastic PDEs perturbed by additive cylindrical noise.
\end{abstract}

%\begin{keywords}
%\kwd{Backward Ornstein-Ulenbeck transition operators}
%, Hamilton-Jacobi equation, mild solution, Gaussian measure in Banach space, Reproducing kernel Hilbert space, null-controllability, non-autonomous stochastic PDEs, cylindrical noise}
%\end{keywords}

\section{Introduction}
Let $\bE$ be a real Banach space and let $T>0$ be fixed. The object of this paper is to study the existence of a \emph{mild solution} $V:[0,T]\times\bE\to\R$ to the following final-value problem for the non-autonomous semi-linear \emph{Hamilton-Jacobi} partial differential equation (HJ-PDE) on $[0,T]\times\bE,$
\begin{equation}\label{hjb0}
\begin{split}
\frac{\partial V}{\partial t}(t,x)+L_tV(t,\cdot)(x)+\cH(t,x,D_xV(t,x))&=0, \ \ \ (t,x)\in [0,T]\times\bE\\
V(T,x)&=\varphi(x).
\end{split}
\end{equation}
The final condition $\varphi:\bE\to\R$ and the nonlinear \emph{Hamiltonian} operator $\cH:[0,T]\times\bE\times\bE^*\to\R$ are given, and for each $t\in[0,T],$ $L_t$ is the second-order differential operator
\[
(L_t\phi)(x):=\seq{-A(t)x,D_x \phi(x)}+\frac{1}{2}\Tr_\bH[G(t)^*D_x^2\phi(x)G(t)], \ x\in D(A(t)).\]
Here $\seq{\cdot,\cdot}$ denotes the duality pairing between $\bE$ and its dual $\bE^*,$ \linebreak $\set{-A(t)}_{t\in[0,T]}$ is a family of densely defined closed linear operators generating a parabolic evolution family on $\bE,$ $\set{G(t)}_{t\in[0,T]}$ is a family of (possibly unbounded) linear operators from a Hilbert space $\bH$ into $\bE,$ $\Tr_\bH[\cdot]$ denotes the trace in $\bH,$ and $D_x\phi(x),D_x^2\phi(x)$ denote first and second order Fr\'{e}chet derivatives of $\phi:\bE\to\R$ at $x\in D(A(t)).$

In this paper we revisit the \emph{mild solution} approach to Hamilton-Jacobi equations initiated by Da Prato \cite{dp0} and Cannarsa \cite{candp}, and continued by Gozzi \cite{gozzi1,gozzi2}, Cerrai \cite{cerrai3,cerrai4} and Masiero \cite{masiero1} (see also Da Prato and Zabcyck \cite{dpz3}, Zabcyck \cite{zabc1} and the references therein). This approach consists in rewriting equation (\ref{hjb0}) in mild-integral form (cf. variation-of-constants formula)
\begin{equation}\label{mild}
V(t,x)=[P(t,T)\varphi](x)+\int_t^T \left[P(t,r)\cH(r,\cdot,D_x V(r,\cdot))\right](x)\,dr, \ \ \ (t,x)\in [0,T]\times \bE
\end{equation}
where $P(s,t)$ is the backward transition evolution operator
\begin{equation}\label{OU-op}
[P(s,t)\varphi](x):=\Exp[\varphi(Z(t))|Z(s)=x], \  x\in\bE, \ t\in [s,T], \ \varphi\in\B_b(\bE)
\end{equation}
associated with the \emph{Ornstein-Uhlenbeck} process $\set{Z(t)}_{t\in [0,T]}$ solution to the non-autonomous stochastic Cauchy problem on $\bE$
\[
dZ(t)+A(t)Z(t)\,dt=G(t)\,dW(t), \ \ t\in [0,T].
\]
Here $\set{W(t)}_{t\in[0,T]}$ is an $\bH$-cylindrical Wiener process defined on a probability space $(\Omega,\Fil,\Prob),$ $\Exp[\cdot]$ denotes expectation in the Bochner-integral sense with respect to the probability measure $\Prob$ and $\B_b(\bE)$ denotes the set of bounded Borel-measurable real-valued maps on $\bE.$

Under the so-called null-controllability condition (see Assumption \ref{assu2} in Section 5 below) the backward transition operators $P(s,t)$ satisfy a strong regularizing property, see Theorem \ref{6.2.2}. For the case in which $\bE$ is a Hilbert space and equation (\ref{hjb0}) is autonomous with respect to time variable (i.e. $A(t)$ and $G(t)$ do not depend on $t$), this regularizing property has been used in conjunction with a fixed point argument to prove existence of a unique solution to the integral equation (\ref{mild}) in a certain space of functions, see e.g. Theorem 9.3 in Zabcyck \cite[Sec. 9]{zabc1}, Da Prato and Zabcyck \cite[Part III]{dpz3} and Masiero \cite{masiero1}.

The main purpose of this paper is to show that this result can be easily generalized to the non-autonomous and Banach-space setting. Namely, we obtain the following (see Theorem \ref{main} below)
\begin{thm*}
Let $\varphi\in\mathcal{C}_b(\bE).$ Suppose Assumptions
{\rm\textbf{(AT)}} and {\rm\textbf{\ref{assu1}}-\textbf{\ref{assu4}}} hold true. Then there exists an unique mild solution to equation $(\ref{hjb0}).$
\end{thm*}
We refer the reader to Sections 4-6 below for the precise statement of Assumptions \textbf{(AT)} and \textbf{\ref{assu1}}-\textbf{\ref{assu4}}. As an example, we consider a non-autonomous HJ equation in $L^p(0,1)$ with $p\geq 2,$ see Example \ref{exHJ:Lp} below.

It should be emphasized that our proof does not present any significant innovation as we follow closely the arguments in the proof for the Hilbert-space case by Masiero \cite[Theorem 2.9]{masiero1}. However, to the best of our knowledge, this is the first paper that deals with infinite-dimensional non-autonomous semi-linear HJ equations in the general Banach-space framework, particularly in Lesbesgue spaces $L^p(\mathcal{O})$ with $p\geq 2.$ This is our main motivation to study HJ equations in a more general Banach-space setting that led to the writing of this paper.

%we illustrate this by formulating a linear parabolic second-order stochastic PDE with time-dependent coefficients perturbed by additive space-time white noise on $[0,T]\times %(0,1),$ as an non-autonomous evolution equation in $L^p(0,1)$ with $p\geq 2$ and verifying that the transition operators of the of a see Examples \ref{ex2} and \ref{ex3} below.

The rest of the paper is organized as follows. In section 2 we recall some basic facts on Gaussian measures in Banach spaces, reproducing kernel Hilbert spaces and the Cameron-Martin formula. We present an alternative proof of a well-known result on regularizing property of Gaussian convolutions which first appeared in the seminal paper by L. Gross \cite{gross1}. In section 3 we review some results from van Neerven and Weis \cite{vnw1} on stochastic integration of deterministic operator valued functions with respect to a cylindrical Wiener process.

In section 4 we recall the setting of Acquistapace and Terreni for parabolic evolution families and non-autonomous evolution equations. In section 5 we introduce backward Ornstein-Uhlenbeck (OU) transition evolution operators in Banach spaces and extend some results from van Neerven \cite[Section 1]{vn1} on the relation between the associated reproducing Kernel Hilbert spaces. In section 6, we state and prove the final result Theorem \ref{main}. Throughout, as the main working example, we consider a linear parabolic second-order stochastic PDE with time-dependent coefficients and space-time white noise formulated as an evolution equation in $L^p(0,1)$ with $p\geq 2.$ We prove the transition operators of the (mild) solution verify the assumptions of the main result. This leads to our final Example \ref{exHJ:Lp}.

\medskip\noindent\emph{Discussion.} Of particular interest are Hamiltonians $\cH$ of the form
\begin{equation}\label{ham}
\cH(t,x,p)=\inf_{u\in M}\set{\seq{F(t,x,u),p}+h(t,x,u)}, \ (t,x,p)\in [0,T]\times\bE\times\bE^*
\end{equation}
where $M$ is a separable metric space, $F:[0,T]\times\bE\times M\to\bE$ and $h:[0,T]\times \bE\times M\to (-\infty,\infty].$ In this case, equation (\ref{hjb0}) is the \emph{Hamilton-Jacobi-Bellman} PDE associated with the dynamic programming principle applied to the following stochastic optimal control problem
\begin{equation}
\text{minimize}  \ \ \ J(X,u)=\Exp\left[\int_0^T h(t,X(t),u(t))\,dt+\varphi(X(T))\right]
\end{equation}
subject to
\begin{itemize}
  \item $\{u(t)\}_{t\in [0,T]}$ is an $M$-valued control process
  \item $\set{X(t)}_{t\in [0,T]}$ is the $\bE$-valued solution to the controlled non-autonomous stochastic evolution equation with additive noise
\begin{equation*}
\begin{split}
dX(t)+A(t)X(t)\,dt&=F(t,X(t),u(t))\,dt+G(t)\,dW(t),\\
X(0)&=x_0\in\bE.
\end{split}
\end{equation*}
\end{itemize}
For the case  in which $\bE$ is Hilbert, under certain additional differentiability assumptions on the Hamiltonian (\ref{ham}), the mild solution of (\ref{hjb0}) can be used to formulate optimality criteria and verification-type results for optimal control problems in Hilbert spaces for stochastic PDEs, see e.g. Da Prato and Zabcyck \cite[Part III]{dpz3} or Masiero \cite[Sec. 4-6]{masiero1}. This can also be combined with Malliavin Calculus and backward stochastic evolution systems in Hilbert spaces to prove  existence of an optimal feedback control, see e.g. Fuhrman and Tessitore \cite{fute2,fute1,fute4,fute3} and the references therein.

Using regularizing properties of stochastic convolutions, Masiero \cite{masiero4} proved existence of mild solutions of a certain class of autonomous HJB equations on the  space of continuous functions $\mathcal{C}(\mathcal{O}).$ Under additional, somewhat restrictive conditions on the nonlinear coefficient $F,$ particularly a dissipative-type condition and a very specific form of dependence with respect to the control variable, Masiero also solved the control problem using backward SDEs but with no use of Malliavin calculus.

At the moment, we are unable to obtain optimality criteria and verification-type results for optimal control problems in Banach spaces for non-autonomous stochastic PDEs as this requires approximation results in $\mathcal{C}_b(\bE)$ by smooth functions that do not seem available at the moment in the general Banach-space setting. However, we believe this can be overcome by employing recent results on Malliavin calculus in Banach spaces (see e.g. Maas \cite{maas}). We will address this issue in a forthcoming paper.

\medskip\noindent\emph{Notation.} Let $\mathcal{O}$ be a bounded domain in $\mathds{R}^d.$ For $m\in\mathds{N}$ and $p\in[1,\infty]$,  $W^{m,p}(\mathcal{O})$ will denote the usual Sobolev space, and for $s\in\R$,  $H^{s,p}(\mathcal{O})$ will denote the scale of spaces
\[
H^{s,p}(\mathcal{O}):=
\left\{
  \begin{array}{ll}
    W^{m,p}(\mathcal{O}), & \ \hbox{if  \ \ $m\in\mathds{N}$;}\\
    \left[W^{k,p}(\mathcal{O}),W^{m,p}(\mathcal{O})\right]_\delta, & \ \text{if} \ \ s\in(0,\infty)\setminus\mathds{N},
  \end{array}
\right.
\]
where $k,m\in\mathds{N}, \ \delta\in(0,1)$ are chosen to satisfy $s=(1-\delta)k+\delta m$ and $[\cdot,\cdot]_\delta$ denotes complex interpolation (see e.g. Triebel \cite{triebel}).

%Since nonlinearities are easier to treat in the space $\mathcal{C}([0,1])$ of real-valued continuous functions on $[0,1]$, the above considerations clearly suggest that we treat the controlled equation in this Banach space.

\section{Gaussian measures in Banach spaces, Cameron-Martin formula and smoothing property}
We recall first some basic facts on Gaussian measures in Banach spaces, particularly the Cameron-Martin formula and the smoothing property of Gaussian convolutions in Banach spaces.

Let $\mathcal{B}(\bE)$ denote the Borel $\sigma-$algebra on the real Banach space $\bE,$ let $\bE^*$ be the continuous dual of $\bE$ and let $\seq{\cdot,\cdot}$ denote the duality pairing between $\bE$ and $\bE^*.$

\begin{defn} A Radon measure $\mu$ on $(\bE,\mathcal{B}(\bE))$ is called \textit{Gaussian} (resp. \textit{centered Gaussian}) if, for any linear functional $x^*\in\bE^*,$ the image measure $\mu\circ\seq{x^*,\cdot}^{-1}$ is a Gaussian (resp. centered
Gaussian) measure on $\R.$
\end{defn}
If $\mu$ is a centered Gaussian measure on $\bE,$ there exists an unique bounded linear operator $Q\in\Lin(\bE^*,\bE)$ called the \textit{covariance operator} of $\mu,$ such that for all $x^*,y^*\in\bE^*$ we have
\[
\seq{Qx^*,y^*}=\int_\bE\seq{x,x^*}\seq{x,y^*}\,\mu(dx).
\]
(see e.g. Bogachev \cite{bogachev1}). Notice that $Q$ is \emph{positive} in the sense that
\[
\seq{Qx^*,x^*}\geq 0  \ \ \forall x^*\in\bE^*,
\]
and \emph{symmetric} in the sense that
\[
\seq{Qx^*,y^*}=\seq{Qy^*,x^*} \ \ \ \forall x^*,y^*\in\bE^*.
\]
The Fourier transform $\hat{\mu}$ of $\mu$ is given by
\[
\hat{\mu}(x^*)=\exp\Bigl(-\frac{1}{2}\seq{Qx^*,x^*}\Bigr),  \ \ \ x^*\in\bE^*.
\]
This identity implies that two centered Gaussian measures are equal whenever their covariance operators are equal.

%The converse is not true: not every positive symmetric $Q\in\Lin(\bE^*,\bE)$ is the covariance operator of a centered Gaussian measure (ref???)
For any $Q\in \Lin(\bE^*,\bE)$ positive and symmetric, the bilinear form on $Q(\bE^*)$ given by
\[
[Qx^*,Qy^*]:=\seq{Qx^*,y^*},    \ \ \ x^*,y^*\in\bE^*.
\]
is a well-defined inner product on $Q(\bE^*).$ We denote with $H_Q$ the Hilbert space completion of $Q(\bE^*)$ with respect to this inner product. The inclusion mapping from $Q(\bE^*)$ into $\bE$ is continuous with respect to the inner product $[\cdot,\cdot]_{H_Q}$ and extends uniquely to a bounded linear injection $i_Q:H_Q\embed\bE.$
\begin{defn}
The pair $(i_Q,H_Q)$ is called the \emph{reproducing kernel Hilbert space} (RKHS) associated with $Q.$
\end{defn}
%Notice that by the definition of $H_Q,$ for $y\in H_Q$ and $x^*\in\bE^*,$
%\begin{equation}\label{HQseq}
%[y,Qx^*]_{H_Q}=\seq{y,x^*}, \ \ \ \ y\in H_Q, \ \ \ .
%\end{equation}
It can be easily shown that the adjoint operator $i_Q^*:\bE^*\to H_Q$ satisfies $i_Q^*x^*=Qx^*$ for all $x^*\in\bE^*.$ Therefore, $Q$ admits the factorization
\[
Q=i_Q\circ i^*_Q.
\]
This factorization immediately implies that $Q$ is weak$^*$-to-weakly continuous and that, if $\bE$
is separable, so is $H_Q.$  We identify for the sake of simplicity $H_Q$ with its image $i_Q(H_Q)\subset \bE.$

%We have the following result,
\begin{prop}[\cite{vn1}, Proposition 1.1]\label{vn1Prop1.1}
Let $Q,\tilde{Q}\in\Lin(\bE^*,\bE)$ be two positive symmetric operators. Then, for the corresponding reproducing kernel Hilbert spaces we have $H_Q\subset H_{\tilde{Q}}$ (as subsets of $\bE)$ if and only if there exist a constant $K>0$ such that
\[
\seq{Qx^*,x^*}\le K\langle\tilde{Q}x^*,x^*\rangle, \ \ \ \ \forall x^*\in\bE^*.
\]
\end{prop}
We will denote with $H_\mu$ (resp. $i_\mu$) instead of $H_Q$ (resp. $i_Q)$ whenever $Q$ is the covariance operator of a Gaussian measure $\mu$ on $\bE.$ In this case, we introduce a linear isometry from $H_\mu$ into $L^2(\bE,\mu)$ as follows: first observe that $\seq{x^*,\cdot}\in L^2(\bE,\mu)$ for every linear functional $x^*\in\bE^*$ and that we have
\begin{equation}\label{var}
\Exp^\mu\abs{\seq{x^*,\cdot}}^2=\int_\bE\abs{\seq{x,x^*}}^2\,\mu(dx)=\seq{Q x^*,x^*},
\ \ \ x^*\in\bE^*.
\end{equation}
Here $\Exp^\mu$ denotes the expectation on the probability space
$(\bE,\B(\bE),\mu).$ Since $Q$ is injective as an operator from
$\bE^*$ into $Q(\bE^*),$ the linear map
\begin{equation}\label{eq3a1}
Q(\bE^*)\ni Q(x^*)\mapsto \seq{x^*,\cdot}\in L^2(\bE,\mu)
\end{equation}
is well-defined and is an isometry in view of (\ref{var}).
\begin{defn}\label{phimu}
We denote by
\begin{equation}\label{eq2a1}
\phi_\mu:H_\mu\to L^2(\bE,\mu)
\end{equation}
the unique extension of the isometry (\ref{eq3a1}) to $H_\mu.$
\end{defn}
This isometry has the property that, for each $h\in H_\mu,$ $\phi_\mu(h)$ is a
$\N(0,|h|^2_{H_\mu})$ random variable. Indeed, for $h\in H_\mu$ fixed,
if $(x_n^*)_n$ is a sequence in $\bE^*$ such that $Qx^*_n\to h$ in
$H_\mu,$ then
\[\seq{x_n^*,\cdot}=\phi_\mu(Qx_n^*)\to \phi_\mu(h), \ \ \ \text{in } L^2(\bE,\mu)\]
and this implies, in particular, that $\Exp^\mu
[e^{i\lambda\seq{x_n^*,\cdot}}]\to \Exp^\mu[e^{i\lambda \phi_\mu(h)}]$ as $n\to\infty$ for all $\lambda\in\R.$ Since $\seq{x_n^*,\cdot}$ is normally
distributed with mean $0$ and variance $\abs{Qx_n^*}_{H_\mu}^2,$ we
have
\[\Exp^\mu
\bigl[e^{i\lambda\seq{x_n^*,\cdot}}\bigr]=\exp\Bigl(-\frac{\lambda^2}{2}\abs{Qx_n^*}^2_{H_\mu}\Bigr),
\ \ \ \lambda\in\R,\] and by dominated convergence, taking the limit as
$n\to\infty$ we get
\[\Exp^\mu\bigl[e^{i\lambda \phi_\mu(h)}\bigr]=\exp\Bigl(-\frac{\lambda^2}{2}\abs{h}^2_{H_\mu}\Bigr), \ \ \ \ \lambda\in\R,\]
which implies that $\phi_\mu(h)$ is a $\N(0,|h|^2_{H_\mu})$-distributed random variable.

\begin{defn}
For each $h\in H_\mu$ we denote by $\mu^h$ the image of the
measure $\mu$ under the translation $z\mapsto z+h,$ that is,
\[\mu^h(A):=\mu(A-h), \ \ \ A\in\B(\bE).\]
We call $\mu^h$ the \textit{shift of the measure $\mu$ by the
vector} $h.$
\end{defn}

\begin{thm}[Cameron-Martin formula]\label{theorema1}
Let $\mu$ be a centered Gaussian measure on $\bE$ with covariance
operator $Q\in\Lin(\bE^*,\bE)$ and let $(i_\mu,H_\mu)$ denote the RKHS
associated with $\mu.$ Then, for any $h\in H_\mu,$ the measure $\mu^h$
is absolutely continuous with respect to $\mu$ and we have
\[\frac{d\mu^h}{d\mu}=\rho_h, \ \ \mu-\text{a.s.}\] with
\[\rho_h:=\exp\Bigl(\phi_\mu(h)-\frac{1}{2}\abs{h}^2_{H_\mu}\Bigr), \ \ \
h\in H_\mu.\]
\end{thm}
\begin{proof}
 See Bogachev \cite[Corollary 2.4.3]{bogachev1}
\end{proof}
For the remainder of this section, we fix $\varphi\in \B_b(\bE)$  and define the mapping $\psi:\bE\to\R$ as
\[
\psi(x):=\int_\bE \varphi(x+z)\,\mu(dz), \ \ \ x\in \bE.
\]
Recall that $\psi:\bE\to\R$ is Fr\'{e}chet differentiable at $x\in\bE$ in the direction of $H_\mu$ if there exists an element of $H_\mu^*,$ denoted by $D_{H_\mu}\psi(x),$ such that
\[
\lim_{\substack{y\in H_\mu\\y\to 0}}\frac{\abs{\psi(x+y)-\psi(x)-\left(D_{H_\mu}\psi(x)\right)(y)}}{\abs{y}_{H_\mu}}=0.
\]
%The following regularizing property was proved in the seminal paper by Leonard Gross on Potential theory in Hilbert spaces \cite[Proposition %9]{gross1} using directly the definition of Fr\'{e}chet derivative. Here we present an alternative proof based on G\^{a}teaux differentiability.
The following regularizing property is a classical result proved by L. Gross in his seminal paper \cite[Proposition 9]{gross1} using directly the notion of Fr\'{e}chet derivative. Here we present an alternative proof based on G\^{a}teaux differentiability.

% We observe first that
%by the Cameron-Martin formula, we have
%\[\psi(h)=\int_\bE \varphi(z)\,\mu^h(dz)=\int_\bE
% \varphi(z)\rho(h,z)\,\mu  (dz), \ \ h\in H_\mu.\]
\begin{prop}\label{sp0}
The map $\psi:\bE\to\R$ is infinitely Fr\'{e}chet differentiable in the direction of $H_\mu.$
The first Fr\'{e}chet derivative of $\psi$ at $x\in\bE$ in the
direction of $y\in H_\mu$ is given by
\begin{equation}\label{ffder}
\left(D_{H_\mu}\psi(x)\right)(y)=\int_\bE \varphi(x+z)\,\phi_\mu(y)(z)\,\mu(dz),
\end{equation}
and the second Fr\'{e}chet derivative of $\psi$ at $x\in\bE$ in the
directions $y_1,y_2\in H_\mu$ is given by
\begin{equation}\label{sfder}
\begin{split}
  \left(D_{H_\mu}^2
 \psi(x)\right)&(y_1,y_2)\\&=-\psi(x)\left[y_1,y_2\right]_{H_\mu}
+\int_\bE \varphi(x+z)\,\phi_\mu(y_1)(z)\,\phi_\mu(y_2)(z)\,\mu(dz).
\end{split}\end{equation}
Moreover we have the estimates
\begin{align}
\norm{D_{H_\mu}\psi(x)}_{H_\mu^*}&\le\abs{\varphi}_0,\\
\norm{D_{H_\mu}^2\psi(x)}_{\Lin(H_\mu,H_\mu^*)}&\le 2\abs{\varphi}_0.
\end{align}
\end{prop}
\begin{proof}
Let us prove first that $\psi$ is G\^{a}teaux differentiable in the direction of $H_\mu,$ i.e.
that for all $x\in\bE$ and $y\in H_\mu,$ the mapping
\[\R\ni\alpha\mapsto\psi(x+\alpha y)\in\R\]
is differentiable at $\alpha=0.$ Let $x\in\bE$ and $y\in H_\mu$ be fixed and let
$\alpha\in\R.$ Observe that by the Cameron-Martin formula, we have
\begin{equation}\label{psi2}
\psi(x+\alpha y)=\int_\bE \varphi(x+z)\,\mu^{\alpha y}(dz)=\int_\bE
 \varphi(x+z)\rho_{\alpha y}(z)\,\mu(dz).
\end{equation}
Since $\phi_\mu(\alpha y)=\alpha\phi_\mu(y)$ in $L^2(\bE,\mu)$ observe that the random variable
\[
\rho_{\alpha y}=\exp\Bigl(\alpha \phi_\mu(y)-\frac{1}{2}\abs{\alpha
 y}^2_{H_\mu}\Bigr)
\]
is defined on a set $\widehat{\bE}=\widehat{\bE}(y)$ of full $\mu$-measure which
depends only on $y,$ for all $\alpha\in\R.$ Thus, the
mapping
\begin{equation}\label{mapg}
g:\R\times\bE\ni(\alpha,z)\mapsto g(\alpha,z):=\rho_{\alpha y}(z)\in\R
\end{equation}
is well-defined and measurable. Moreover, for $\varepsilon>0$ fixed we have the
following estimate for all $\abs{\alpha_0}<\varepsilon, \
z\in\widehat{\bE},$
\begin{align}
\abs{\frac{\partial
g}{\partial\alpha}(\alpha_0,z)}&=\rho(\alpha_0
 y,z)\abs{\phi_\mu(y)(z)-\alpha_0\abs{y}^2_{H_\mu}}\notag\\
&\le\exp(\varepsilon\abs{\phi_\mu(y)(z)})
\left(\abs{\phi_\mu(y)(z)}+\varepsilon\abs{y}^2_{H_\mu}\right).\label{est1}
\end{align}
We know $\phi_\mu(y)$ is Gaussian random variable with moment generating
function
\[
\Exp^\mu\bigl[e^{\lambda
 \phi_\mu(y)}\bigr]=\exp\Bigl(\frac{\lambda^2}{2}\abs{y}^2_{H_\mu}\Bigr), \ \ \ \lambda\in\R.
\]
This implies, in particular, that
$\exp(\varepsilon\abs{\phi_\mu(y)})%\footnote{if $X\sim\N(0,\sigma^2)$
%then $\abs{X}/\sigma\sim\chi(1)$ and its moment generating function also exists}
$ belongs to $L^2(\bE,\mu).$ Since $\phi_\mu(y)\in L^2(\bE,\mu),$ by H\"{o}lder's inequality the right hand side
in (\ref{est1}) belongs to $L^1(\bE,\mu).$ Thus we may differentiate in the right hand-side of (\ref{psi2})
%the function $g$ satisfies the conditions of Lemma \ref{brz} in the Appendix with $T=\bE$ and $T_1=\widehat{\bE}$
with respect to $\alpha$ under the  sign and obtain that the G\^{a}teaux derivative of $\psi$ at $x$ in the direction of $y$ is given by
\begin{align*}
(d_{H_\mu}\psi(x))(y)&=\Bigl.\frac{d}{d\alpha}\Bigr|_{\alpha=0}\psi(x+\alpha
y)\\
&=\int_\bE\varphi(x+z)\left[\Bigl.\frac{\partial}{\partial\alpha}\Bigr|_{\alpha=0}\rho_{\alpha y}(z)\right]\,\mu(dz),\\
&=\int_\bE\varphi(x+z)\phi_\mu(y)(z)\,\mu(dz),
\end{align*}
as well as the following estimate
\[\norm{d_{H_\mu}\psi(x)}_{\Lin(H_\mu,\R)}\le\abs{\varphi}_0.\]
In turn this implies that the Gateaux derivative $d\psi:H_\mu\to\Lin(H_\mu,\R)$ is continuous and uniformly bounded.
Since $\psi$ is also continuous and uniformly bounded on $H_\mu,$ by
Theorem 3 in Aronszajn \cite[Ch 2, Section 1]{aronszajn} we conclude that
$\psi$ is Fr\'{e}chet differentiable in the direction of $H_\mu$ and (\ref{ffder}) follows.

For the second-order G\^{a}teaux derivative, if $y_1,y_2\in H_\mu$ and $\alpha\in\R$ we have
\begin{align*}
(d_{H_\mu}\psi(x+\alpha y_2))(y_1)&=\int_\bE\varphi(x+\alpha y_2+z)\phi_\mu(y_1)(z)\,\mu(dz)\\
&=\int_\bE\varphi(x+\xi)\phi_\mu(y_1)(\xi-\alpha y_2)\,\mu^{\alpha y_2}(d\xi)\\
&=\int_\bE\varphi(x+\xi)\phi_\mu(y_1)(\xi-\alpha y_2)\,\rho_{\alpha y_2}(\xi)\,\mu(d\xi)
\end{align*}
where we have used again the Cameron-Martin formula and the change of variable $\xi=z+\alpha y_2$ whose push-forward measure with respect with $\mu$ is given by $\mu^{\alpha y_2}.$

If $y_1=Qx_1^*$ for some $x_1^*\in\bE^*,$ from the definition of $\phi_\mu$ it follows that
\[\phi_\mu(y_1)(\xi-\alpha y_2)=\seq{x_1^*,\xi-\alpha y_2}=\seq{x_1^*,\xi}-\alpha\seq{x_1^*,y_2}=\phi_\mu(y_1)(\xi)-\alpha[y_1,y_2]_{H_\mu}\]
in which case we have
\begin{equation}\label{dy2y1}
(d_{H_\mu}\psi(x+\alpha y_2))(y_1)=\int_\bE\varphi(x+\xi)\left(\phi_\mu(y_1)(\xi)-\alpha[y_1,y_2]_{H_\mu}\right)\rho_{\alpha y_2}(\xi)\,\mu(d\xi).
\end{equation}
Since both sides of (\ref{dy2y1}) are continuous in $y_1\in H_\mu$ and $Q(\bE^*)$ is dense in $H_\mu,$ the above equality holds for any $y_1\in H_\mu.$ In addition, the equality
\begin{align*}
\Bigl.\frac{\partial}{\partial\alpha}\Bigr|_{\alpha=0}&\left[\left(\phi_\mu(y_1)(\xi)-\alpha\left[y_1,y_2\right]_{H_\mu}\right)\rho(\alpha
y_2,\xi)\right]\\&=-\left[y_1,y_2\right]_{H_\mu}+\phi_\mu(y_1)(\xi)\phi_\mu(y_2)(\xi),
\end{align*}
holds for all $\xi$ in a subset of $\bE$ with full $\mu$-measure that only depends
on $y_2.$ Again, we can differentiate under the integral sign with respect to $\alpha$ to obtain the second G\^{a}teaux derivative
of $\psi$ at $x$ in the direction o $y_1$ and $y_2,$
\begin{align*}
\left(d_{H_\mu}^2 \psi(x)\right)&(y_1,y_2)=\Bigl.\frac{d}{d\alpha}\Bigr|_{\alpha=0}(d_{H_\mu}\psi(x+\alpha y_2))(y_1)\\
&=\int_\bE\varphi(x+\xi)\Bigl.\frac{\partial}{\partial\alpha}\Bigr|_{\alpha=0}\left[\left(\phi_\mu(y_1)(\xi)
-\alpha\left[y_1,y_2\right]_{H_\mu}\right)\rho_{\alpha y_2}(\xi)\right]\,\mu(d\xi)\\
&=\int_\bE\varphi(x+\xi)\left(\phi_\mu(y_1)(\xi)\phi_\mu(y_2)(\xi)-\left[y_1,y_2\right]_{H_\mu}\right)\,\mu(d\xi)
\end{align*}
together with the following estimate
\[\norm{d_{H_\mu}^2\psi(x)}_{\Lin(H_\mu,H_\mu^*)}\le 2\abs{\varphi}_0,\] for all
$x\in\bE.$ By the same argument as above $\psi$ is also twice Fr\'{e}chet differentiable and
(\ref{sfder}) follows.
\end{proof}
By identifying $H_\mu$ with its dual $H^*_\mu,$ the map $D_{H_\mu}^2\psi(x)$ defines a
bounded linear operator on $H_\mu.$ The following lemma shows that it is actually a Hilbert-Schmidt
operator. The proof follows the same argument as in the Hilbert-space case (see e.g. \cite[Chapter 3]{dpz3}). We include the proof for the sake of completeness.

\begin{lem}\label{6.2.7}
For each $x\in\bE$ we have $D_{H_\mu}^2\psi(x)\in \T_2(H_\mu)$ and
\begin{equation}\label{hs2der0}
\norm{D_{H_\mu}^2\psi(x)}_{\T_2(H_\mu)}\le \sqrt{2}\abs{\varphi}_0.
\end{equation}
If $\varphi\in C^1_b(\bE)$ we have
\begin{equation}\label{hs2der1}
\norm{D_{H_\mu}^2\psi(x)}_{\T_2(H_\mu)}\le \abs{\varphi}_1.
\end{equation}
\end{lem}

\begin{proof}
Let $(e_i)_i$ be an orthonormal basis of $H_\mu$ and let $x\in\bE$
be fixed. Let us prove first the case $\varphi\in C^1_b(\bE).$ By
the same argument used in the proof of (\ref{ffder}) one can
derive
\[[D_{H_\mu}^2\psi(x)y_1,y_2]=\int_\bE [D\varphi(x+z),y_1]_{H_\mu}\,\phi_\mu(y_2)(z)\,\mu(dz), \ \ \ y_1,y_2\in H_\mu.\]
Since the map $\phi_\mu$ is an isometry from $H_\mu$ to $L^2(\bE,\mu),$ the
random variables $\phi_\mu(e_k), k\in\mathbb{N},$ form a complete
orthonormal system in $L^2(\bE,\mu)$ and by Parseval identity and
dominated convergence we get
\begin{align*}
\norm{D_{H_\mu}^2\psi(x)}_{\T_2(H_\mu)}^2&=\sum_{i=1}^\infty\abs{D_{H_\mu}^2\psi(x)e_i}_{H_\mu}^2
=\sum_{i,k=1}^\infty\abs{[D_{H_\mu}^2\psi(x)e_i,e_k]_{H_\mu}}^2\\
&=\sum_{i,k=1}^\infty\abs{\seq{[D\varphi(x+\cdot),e_i]_{H_\mu},\phi_\mu(e_k)}_{L^2(\bE,\mu)}}^2\\
&=\sum_{i=1}^\infty\norm{[D\varphi(x+\cdot),e_i]_{H_\mu}}_{L^2(\bE,\mu)}^2\\
&=\int_\bE\sum_{i=1}^\infty\abs{[D\varphi(x+z),e_i]_{H_\mu}}^2\,\mu(dz)\\
&=\int_\bE\abs{D\varphi(x+z)}^2_{H_\mu}\,\mu(dz)\\
&\le\abs{\varphi}_1^2
\end{align*}
and (\ref{hs2der1}) follows. For the general case $\varphi\in\B_b(\bE),$ we define the random variables
\[\zeta_{i,k}:=
  \begin{cases}
   \frac{1}{\sqrt{2}}\left(\phi_\mu(e_i)^2-1\right), & \text{if} \ i=k,\\
    \phi_\mu(e_i)\phi_\mu(e_k), & \text{if} \ i\neq k.
  \end{cases}\]
Since $\phi_\mu(e_k), k\in\mathbb{N},$ are independent Gaussian random
variables with mean $0$ and variance $1,$ we get
\[\seq{\zeta_{i,k},\zeta_{i',k'}}_{L^2(\bE,\mu)}=0, \ \ \ \text{for } (i,k)\neq(i',k')\]
and
\begin{align*}
\norm{\zeta_{i,k}}_{L^2(\bE,\mu)}^2&=\Exp\zeta_{i,k}^2=\Exp\left(\phi_\mu(e_i)^2\phi_\mu(e_k)^2\right)=1, \ \ \ i\neq k\\
\norm{\zeta_{i,i}}_{L^2(\bE,\mu)}^2&=\Exp\zeta_{i,i}^2=\frac{1}{2}\Exp\left(\phi_\mu(e_i)^4-2\phi_\mu(e_i)^2+1\right)=\frac{1}{2}(3-2+1)=1,
\end{align*}
i.e. the system $\{\zeta_{i,k}: i,k\in\mathbb{N}\}$ is orthonormal
in $L^2(\bE,\mu).$ Recalling (\ref{sfder}), for $i,k\in\mathbb{N}$
we have
\[[D_{H_\mu}^2\psi(x)e_i,e_k]_{H_\mu}=\left\{
  \begin{array}{ll}
   \sqrt{2}\seq{\beta,\zeta_{i,i}}_{L^2(\bE,\mu)} , & \hbox{if} \ i=k,\\
    \seq{\beta,\zeta_{i,k}}_{L^2(\bE,\mu)}, & \hbox{if} \ i\neq k,
  \end{array}
\right.\] where $\beta(z):=\varphi(x+z).$ Thus, from the Parseval
identity and Bessel inequality it follows that
\begin{align*}
\norm{D_{H_\mu}^2\psi(x)}_{\T_2(H_\mu)}^2&=\sum_{i=1}^\infty\abs{D_{H_\mu}^2\psi(x)e_i}_{H_\mu}^2
=\sum_{i,k=1}^\infty\abs{[D_{H_\mu}^2\psi(x)e_i,e_k]_{H_\mu}}^2\\
&=2\sum_{i=1}^\infty\abs{\seq{\beta,\zeta_{i,i}}_{L^2(\bE,\mu)}}^2+\sum_{\substack{i,k=1\\
i\neq k}}^\infty\abs{\seq{\beta,\zeta_{i,k}}_{L^2(\bE,\mu)}}^2\\
&\le
2\sum_{i,k=1}^\infty\abs{\seq{\beta,\zeta_{i,k}}_{L^2(\bE,\mu)}}^2\\
&\le 2 \norm{\beta}^2_{L^2(\bE,\mu)}\\&\le 2\abs{\varphi}^2_0.
\end{align*}
\end{proof}

\section[Stochastic integration of deterministic operator-valued functions in Banach spaces]{Stochastic integration of deterministic
operator-valued\\functions in Banach spaces}

In this section we review some of the results from
van Neerven and Weis \cite{vnw1} on stochastic integration of deterministic operator
valued functions with respect to a cylindrical Wiener process.
From this point onwards $(\bH,[\cdot,\cdot]_\bH)$ denotes a
Hilbert space and $(\gamma_n)_n$ a sequence of real-valued
standard Gaussian random variables on a probability space $(\Omega
,\Fil,\Prob)$ endowed with a filtration
$\mathds{F}=\{\Fil_t\}_{t\ge 0}.$

\begin{defn}\label{cylwiener}
A family $W(\cdot)=\{W(t)\}_{t\geq 0}$ of bounded linear operators from $\bH$ into $L^2(\Omega;\R)$ is called a $\bH-$\emph{cylindrical Wiener process}
(with respect to the filtration $\mathds{F})$ iff
\begin{enumerate}[(i)]
\item $\Exp\, W(t)y_1 W(t)y_2 = t[y_1 ,y_2]_\bH$, for all $t\ge 0$ and $y_1 ,y_2 \in\bH$,

\item for each $y\in\bH$, the process $\{W(t)y\}_{t\geq 0}$ is a standard one-dimensional Wiener process with respect to $\mathds{F}.$
\end{enumerate}
\end{defn}

\begin{defn}
$R\in\Lin(\bH,\bE)$ is $\gamma-$\emph{radonifying} iff there exists an orthonormal basis $(e_n)_{n\geq 1}$ of $\bH$ such that the sum $\sum_{n\geq 1}\gamma_n Re_n$ converges in $L^2(\Omega;\bE).$
\end{defn}

We denote by $\gamma(\bH,\bE)$ the class of $\gamma-$radonifying operators from $\bH$ into $\bE,$ which can be proved to be a Banach space when equipped with the norm
\[
\norm{R}^2_{\gamma(\bH,\bE)}:=\Exp\Bigl|\sum_{n\geq 1}\gamma_n Re_n\Bigr|^2_\bE , \ \ \ \ R\in \gamma(\bH,\bE).
\]
The above definition is independent of the choice of the orthonormal basis $(e_n)_{n\geq 1}$ of $\bH.$ Moreover, $\gamma(\bH,\bE)$ is continuously embedded into $\Lin(\bH,\bE)$ and is an operator ideal in the sense that if $\bH'$ and $\bE'$ are Hilbert and Banach spaces respectively such that $S_1\in\Lin(\bH',\bH)$ and $S_2\in\Lin(\bE,\bE')$ then $R\in \gamma(\bH,\bE)$ implies $S_2RS_1\in\gamma(\bH',\bE')$ with
\[
\norm{S_2RS_1}_{\gamma(\bH',\bE')}\le \norm{S_2}_{\Lin(\bE,\bE')}\norm{R}_{\gamma(\bH,\bE)}\norm{S_1}_{\Lin(\bH',\bH)}
\]
It can also be proved that $R\in \gamma(\bH,\bE)$ iff $RR^*$ is the covariance operator of a centered Gaussian measure on $\B(\bE),$ and if $\bE$ is a Hilbert space, then $R\in\gamma(\bH,\bE)$ iff $R$ is a Hilbert-Schmidt operator from $\bH$ into $\bE$ (see e.g. van Neerven \cite{vn0} and the references therein).

%There is also a very useful characterization of $\gamma-$radonifying operators if $\bE$ is a $L^p-$space,
%\begin{lem}[\cite{vnvw}, Lemma 2.1]\label{gammalp}
%Let $(S,\mathfrak{A},\rho)$ be a $\sigma-$finite measure space and let $p\geq 1.$ Then, for an operator $R\in\Lin(\bH,L^p(S))$ the following %assertions are equivalent
%\begin{enumerate}
%  \item $R\in\gamma(\bH,L^p(S)),$
%  \item There exists a function $g\in L^p(S)$ such that for all $y\in\bH$ we have
%  \[
%  \abs{(Ry)(s)}\le \abs{y}_\bH\cdot g(s), \ \ \ \rho-\mbox{a.e.} \ s\in S.
%  \]
%\end{enumerate}
%In such situation, there exists a constant $c>0$ such that $\norm{R}_{\gamma(\bH,L^p(S))}\le c\abs{g}_{L^p(S)}.$
%\end{lem}

%\begin{ex}

%\end{ex}
The $\gamma-$radonifying property in the following example goes back to Brzezniak \cite{brz96}, and will be used later in our main Example \ref{ex2}. For the sake of completeness, we include a proof which follows closely arguments from van Neerven \cite[Chapter 15]{vn0}.
\begin{ex}\label{ex0}
For $p\geq 1,$ let $\Delta_p$ denote the realization of $-\frac{d^2}{d\xi^2}$ in $L^p(0,1)$ with zero-Dirichlet boundary conditions.
Then, for $\sigma\in (\frac{1}{4},1),$  the identity operator on $D(\Delta_2)$ extends to a continuous embedding $j:D(\Delta_2)\embed D(\Delta_p^{1-\sigma})$ that is $\gamma-$radonifying.
\end{ex}

\begin{proof}
The functions $e_n(\xi)=\sqrt{2}\sin(n\pi\xi),$ $n\geq 1,$ form an orthonormal basis of eigenfunctions for $\Delta_2$ with eigenvalues $\lambda_n=(n\pi)^2.$ If we endow $D(\Delta_2)$ with the equivalent Hilbert norm $\abs{y}_{D(\Delta_2)}:=\abs{\Delta_2 y}_{L^2(0,1)},$ the functions $\lambda_n^{-1}e_n$ form an orthonormal basis for $D(\Delta_2).$

Let $(\gamma_n)_n$ be a Gaussian sequence on a probability space $(\Omega,\Fil,\Prob).$ Then, we have
\begin{align}
\Exp\Bigl|\sum_{n\geq 1}\gamma_n\lambda_n^{-1}e_n\Bigr|^2_{D(\Delta_p^{1-\sigma})}
&=\Exp\Bigl|\sum_{n\geq 1}\gamma_n\lambda_n^{-1}\Delta_p^{1-\sigma}e_n\Bigr|^2_{L^p(0,1)}\notag\\
&=\Exp\Bigl|\sum_{n\geq 1}\gamma_n(n\pi)^{-2\sigma}e_n\Bigr|^2_{L^p(0,1)}\label{gammanen}
\end{align}
Using H\"{o}lder's inequality, we have
\begin{align*}
\Exp\Bigl|\sum_{n=N}^M\gamma_n(n\pi)^{-2\sigma}e_n\Bigr|^2_{L^p(0,1)}
&\le \left(\Exp\Bigl|\sum_{n=N}^M\gamma_n(n\pi)^{-2\sigma}e_n\Bigr|^p_{L^p(0,1)}\right)^{2/p}\\
&= \left(\Exp\int_0^1\bigl|\sum_{n=N}^M\gamma_n(n\pi)^{-2\sigma}e_n(\xi)\bigr|^p\,d\xi\right)^{2/p}\\
&= \left(\int_0^1 \Exp\bigl|\sum_{n=N}^M\gamma_n(n\pi)^{-2\sigma}e_n(\xi)\bigr|^p\,d\xi\right)^{2/p}\\
&\le\left(\int_0^1 \Bigl(\Exp\bigl|\sum_{n=N}^M\gamma_n(n\pi)^{-2\sigma}e_n(\xi)\bigr|^2\Bigr)^{p/2}\,d\xi\right)^{2/p}
\end{align*}
By Kahane-Khintchine inequality, there exists a constant $c'$ such that
\begin{align*}
  \Exp\bigl|\sum_{n=N}^M\gamma_n(n\pi)^{-2\sigma}e_n(\xi)\bigr|^2&\le c'\sum_{n=N}^M\abs{(n\pi)^{-2\sigma}e_n(\xi)}^2\\
  &=c'\sum_{n=N}^M (n\pi)^{-4\sigma}e^2_n(\xi)
\end{align*}
Hence, we obtain
\begin{align*}
\Exp\Bigl|\sum_{n=N}^M\gamma_n(n\pi)^{-2\sigma}e_n\Bigr|^2_{L^p(0,1)}
&\le  \left(\int_0^1 \Bigl[c'\sum_{n=N}^M (n\pi)^{-4\sigma}e^2_n(\xi)\Bigr]^{p/2}\,d\xi\right)^{2/p}\\
&= c'\Bigl|\sum_{n=N}^M (n\pi)^{-4\sigma}e_n^2\Bigr|_{L^{p/2}(0,1)}\\
&\le c'\sum_{n=N}^M (n\pi)^{-4\sigma}\abs{e_n^2}_{L^{p/2}(0,1)}
\end{align*}
Since $\abs{e_n^2}_{L^{p/2}(0,1)}=\abs{e_n}_{L^{p}(0,1)}^2\le 2$ for all $n\geq 1,$ it follows that
\[
\Exp\Bigl|\sum_{n=N}^M\gamma_n(n\pi)^{-2\sigma}e_n\Bigr|^2_{L^p(0,1)}\le 2c'_p\sum_{n=N}^M (n\pi)^{-4\sigma}.
\]
The right-hand side of the last inequality tends to $0$ as $N,M\to\infty$ since $\sigma>\frac{1}{4}.$ Therefore, the right-hand side of (\ref{gammanen}) is finite, and the claim follows.
\end{proof}

Before we discuss the integral for $\Lin(\bH,\bE)-$valued functions, we observe that we can integrate certain $\bH-$valued functions with respect to a $\bH-$cylindrical Wiener process $W(\cdot).$ For a step function of the form $\psi=\mathbf{1}_{(s,t]}y$ with $y\in\bH$ we define
\[
\int_0^T\psi(r)\,dW(r):=W(t)y-W(s)y.
\]
This extends to arbitrary step functions $\psi$ by linearity, and a standard computation shows that
\[
\Exp\Bigl|\int_0^T \psi(r)\,dW(r)\Bigr|_\R^2=\int_0^T\abs{\psi(t)}_\bH^2\,dt.
\]
Since the set of step functions $L_{\mathrm{step}}^2(0,T;\bH)$ is dense in $L^2(0,T;\bH),$ the map
\[
I_T:L_{\mathrm{step}}^2(0,T;\bH)\ni\psi\mapsto \int_0^T\psi(t)\,dW(t)\in L^2(\Omega;\R)
\]
extends to a (linear!) isometry from $L^2(0,T;\bH)$ into $L^2(\Omega).$ We now define the stochastic integral for deterministic $\Lin(\bH,\bE)-$valued functions with respect to $W(\cdot).$ %As before, $\bE$ denotes a real Banach space.
\begin{defn}
\begin{enumerate}
\item A function $\Phi:(0,T)\to\Lin(\bH,\bE)$ is said to belong \emph{scalarly} to $L^2(0,T;\bH)$ if the map
\[
[0,T]\ni t\mapsto\Phi(t)^*x^*\in\bH\
\]
belongs to $L^2(0,T;\bH)$ for every $x^*\in\bE^*.$

\smallskip\item A function $\Phi:(0,T)\to\Lin(\bH,\bE)$ is said to be \emph{stochastically integrable} with respect to $W(\cdot)$ if it belongs scalarly to $L^2(0,T;\bH)$ and for all $A\subset (0,T)$ measurable there exists a random variable $Y_A\in L^2(\Omega,\Fil,\Prob;\bE)$ such that
\[
\seq{Y_A,x^*}=\int_0^T \mathbf{1}_A (t)\Phi(t)^*x^*\,dW(t), \ \ \ \Prob-\text{a.s.}, \ \ \text{for all} \ x^*\in\bE^*.
\]
We denote
\[
\int_A \Phi(t)\,dW(t):=Y_A
\]
\end{enumerate}
\end{defn}

By Fernique's theorem, the $\bE-$valued random variables $Y_A$ are uniquely determined almost everywhere and Gaussian. In particular $Y_A\in L^p(\Omega;\bE)$ for all $p\geq 1.$

For a function $\Phi:(0,T)\to\Lin(\bH,\bE)$ that belongs scalarly to $L^2(0,T;\bH)$ we define an operator $R_\Phi:L^2(0,T;\bH)\to \bE^{**}$ by
\[
\seq{x^*,R_\Phi f}:=\int_0^T [\Phi(t)^*x^*,f(t)]_\bH\,dt, \ \ \ f\in L^2(0,T;\bH), \ \ x^*\in\bE^*.
\]
Observe that $I_\Phi$ is the adjoint of the operator
\[
\bE^*\ni x^*\mapsto \Phi(t)^*x^*\in L^2(0,T;\bH).
\]
If $\Phi(\cdot)y$ is strongly measurable for all $y\in\bH$ then $R_\Phi$ maps $L^2(0,T;\bH)$ into $\bE.$ The following theorem characterizes the class of stochastically integrable functions with respect to the $\bH-$cylindrical Wiener process $W(\cdot).$

\begin{thm}[\cite{vnw1}, Theorem 4.2]\label{vnw1thm4.2}
Let $\bE$ be a separable Banach space. For a function $\Phi:(0,T)\to\Lin(\bH,\bE)$ that belongs scalarly to $L^2(0,T;\bH)$ the following assertions are equivalent
\begin{enumerate}
\item $\Phi$ is stochastically integrable with respect to $W(\cdot)$

\item There exists an $\bE-$valued random variable $Y$ such that for all $x\in\bE^*$
    \[
    \seq{Y,x^*}=\int_0^T \Phi(t)^*x^*\,dW(t), \ \ \ \Prob-\text{a.s.}, \ \ \forall x^*\in\bE^*
    \]
\item There exists a centered Gaussian measure $\mu$ on $\bE$ with covariance operator $Q\in\Lin(\bE,\bE^*)$ such that for all $x\in\bE^*$
    \[
    \seq{Qx^*,x^*}=\int_0^T \abs{\Phi(t)^*x^*}^2_\bH\,dt;
    \]
\item There exist a separable Hilbert space $\mathfrak{H}$ a linear bounded operator $S\in\gamma(\mathfrak{H},\bE)$ such that for all $x\in\bE^*$
    \[
    \int_0^T\abs{\Phi(t)^*x^*}^2_\bH\,dt\le \abs{S^*x^*}^2_{\mathfrak{H}}
    \]

\item $R_\Phi$ maps $L^2(0,T;\bH)$ into $\bE$ and $R_\phi\in\gamma(L^2(0,T;\bH);\bE).$
\end{enumerate}
Moreover, for all $y\in\bH$ the function $\Phi(\cdot)y$ is stochastic integrable with respect to $W(\cdot)y$ and we have the series representation
\[
\int_0^T\Phi(t)\,dW(t)=\sum_{n=1}^\infty\int_0^T\Phi(t)e_n\,dW(t)e_n
\]
where $(e_n)_{n\geq 1}$ is any orthonormal basis for $\bH.$ The series converges $\Prob-$a.s and in $L^p(\Omega;\bE)$ for all $p\in [0,\infty).$ The measure $\mu$ is the distribution of $\int_0^T\Phi(t)\,dW(t)$ and we have the isometry
\[
\Exp\Bigl|\int_0^T\Phi(t)\,dW(t)\Bigr|_\bE=\norm{R_\Phi}_{\gamma(L^2(0,T;\bH);\bE)}
\]
\end{thm}

We conclude this section with a sufficient condition for stochastic integrability in spaces of type 2 (see e.g. van Neerven and Weis \cite[Theorem 4.7]{vnw1} or  \cite[Theorem 5.1]{vnw2}).
\begin{defn}
$\bE$ is said to be of \emph{type} $2$ iff there exists $K_2>0$ such that
\begin{equation}
 \Exp\Bigl| \sum_{i=1}^n \epsilon_i x_i\Bigr|_\bE^2 \le K_2\sum_{i=1}^n | x_i |_\bE^2
\end{equation}
for any finite sequence $\{x_i\}_{i=1}^n$ of elements of $\bE$ and for any finite sequence $\{\epsilon_i\}_{i=1}^n$ of $\{-1,1\}-$valued symmetric i.i.d. random variables.
\end{defn}

\begin{thm}\label{vnw1thm4.7}
Let $\bE$ be a separable real Banach space of type 2. If $\Phi:(0,T)\to \Lin(\bH,\bE)$ belongs scalarly to $L^2(0,T;\bH),$ for almost all $t\in(0,T)$ we have $\Phi(t)\in\gamma(\bH,\bE),$ and
\[
\int_0^T \norm{\Phi(t)}_{\gamma(\bH,\bE)}^2\,dt<\infty
\]
then $\Phi$ is stochastically integrable with respect to $W(\cdot)$ and
\[
\Exp\Bigl|\int_0^T\Phi(t)\,dW(t)\Bigr|^2_\bE\le K_2^2\int_0^T \norm{\Phi(t)}_{\gamma(\bH,\bE)}^2\,dt.
\]
\end{thm}
\begin{proof}
See Theorem 5.1 in van Neerven and Weis \cite{vnw2}.
\end{proof}

\section{Parabolic evolution families}
Since there is no unified theory for parabolic evolution families and non-autonomous evolution equations, we restrict in this paper to the class of parabolic problems and setting introduced by P. Acquistapace and B. Terreni in \cite{AcquiTer87}. We start this section by recalling the definition of positive operators with bounded imaginary powers.

\begin{defn}
Let $A$ be a closed, densely defined linear operator on a Banach space $\bE.$ We say that $A$ is \emph{positive} if $(-\infty,0]\subset \rho(A)$ and there exists $C\geq 1$ such that
\[
\Vertt (\lambda I+A)^{-1}\Vertt_{\Lin(\bE)} \le  \frac{C}{1+\lambda}, \ \ \ \text{ for all }\lambda\geq 0.
\]
\end{defn}

\begin{rem}
It is well known that if $A$ is a positive operator on $\bE$,  then $A$ admits (not necessarily bounded) fractional powers $A^z$ of any order $z\in\mathds{C}$  (see e.g. Amann \cite[Chapter III, Section 4.6]{amann}). Recall that, in particular, for $\abs{\Re z}\le 1$ the fractional power $A^z$ is the closure of the linear mapping
\[
D(A)\ni x\mapsto \frac{\sin \pi z}{\pi z}\int_0^{+\infty} t^z(tI+A)^{-2}Ax\,dt\in\bE,
\]
Moreover, if $\Re z\in (0,1)$,  then $A^{-z}\in\Lin(\bE)$ and we have
\[
A^{-z}x=\frac{\sin \pi z}{\pi}\int_0^{+\infty}t^{-z}(tI+A)^{-1}x\,dt
\]
see e.g. Amann \cite[p. 153]{amann}.
\end{rem}
\begin{defn}\label{defbip}
The class $\BIP(\theta,\bE)$ of operators with \emph{bounded imaginary powers} on $\bE$ with parameter $\theta\in[0,\pi)$ is defined as the class of positive operators $A$ on $\bE$ with the property that $A^{is}\in \Lin(\bE)$ for all $s\in\R$ and there exists a constant $K>0$ such that
\[
\Vertt A^{is} \Vertt_{\Lin(\bE)} \le K e^{\theta |s|}, \; s \in \R.
\]
\end{defn}
For each $t\in [0,T]$ let $A(t)$ be a densely defined closed linear operator on a Banach space $\bE.$ For each $s\in[0,T],$ consider the following non-autonomous Cauchy problem
\begin{equation}\label{NACP1}
\begin{split}
y'(t)+A(t)y(t)&=0, \ \ \ t\in (s,T]\\
y(s)&=x\in\bE.
\end{split}
\end{equation}
\begin{defn}
We say that $y\in\mC((s,T];\bE)\cap\mC^1((s,T];\bE)$ is a \emph{classical} solution of (\ref{NACP1}) if $y(t)\in D(A(t))$ for all $t\in (s,T]$ and (\ref{NACP1}) holds.
\end{defn}
\begin{defn}
We say that a classical solution $y$ of (\ref{NACP1}) is also a \emph{strict} solution  if in addition $y\in \mC^1([s,T];\bE),$ $x\in D(A(s))$ and $A(t)y(t)\to A(s)x$ as $t\to s.$
\end{defn}
We say that condition \textbf{(AT)} is satisfied if the two following conditions hold

\noindent\textbf{(AT1)} \emph{There exist constants $w\in\R, \ K\geq 0$ and $\phi\in(\frac{\pi}{2},\pi)$ such that
  \[
  \Sigma(\phi,w):=\set{w}\cup\set{\lambda\in\mathds{C}\setminus\set{w}:\abs{\arg(\lambda-w)}\le\phi}\subset \rho(-A(t))
  \]
  and for all $\lambda\in\Sigma(\phi,w)$ and $t\in [0,T],$
\[
|\!|(A(t)+\lambda I)^{-1}|\!|_{\Lin(\bE)}\le\frac{K}{1+\abs{\lambda-w}}.
\]
}
\noindent\textbf{(AT2)} \emph{There exist constants $L \geq 0$ and $\mu, \nu \in (0,1)$ with $\mu+\nu>1$ such that for all $\lambda\in\Sigma(\phi,0)$ and $s,t\in[0,T],$
\[
\left|\!\left|A_w(t)(A_w(t)+\lambda I)^{-1}[A_w(t)^{-1}-A_w(s)^{-1}]\right|\!\right|_{\Lin(\bE)}\le L\frac{\abs{t-s}^\mu}{(\abs{\lambda}+1)^\nu},
\]
where $A_w(t):=A(t)+wI.$
}

Operators satisfying \textbf{(AT1)} are called \emph{sectorial} (of type $(\phi,K,w)$). Equation (\ref{NACP1}) is called parabolic because of the sectoriality of the operators $A(t).$

If Assumption \textbf{(AT1)} is satisfied and the domains are constant i.e. $D(A(t))=D(A(0))$ for all $t\in [0,T],$ and the map $[0,T]\ni t\mapsto A(t)\in \Lin(D(A(0)),\bE)$ is H\"{o}lder continuous with exponent $\eta,$ then \textbf{(AT2)} is satisfied with $\mu=\eta$ and $\nu=1,$ see e.g. Acquistapace and Terreni \cite[Section 7]{AcquiTer87}. In this case such conditions reduce to the theory of Sobolevskii and Tanabe for constant domains (see e.g. Pazy \cite{pazy} or Tanabe \cite{tanabe}).

\begin{defn}
In what follows we denote $\mathfrak{T}:=\set{(t,s)\in[0,T]^2:s\le t}.$ A family of bounded operators $\{S(t,s)\}_{(t,s)\in\mathfrak{T}}$
on $\bE$ is called a \emph{strongly continuous evolution family } if the following hold
\begin{enumerate}
\item $S(t,t)=I,$ for all $t\in [0,T].$
\item $S(t,s)=S(t,r)S(r,s)$ for all $0\le s\le r\le t\le T.$
\item The mapping $\mathfrak{T}\ni(t,s)\mapsto S(t,s)\in\Lin(\bE)$ is strongly continuous.
\end{enumerate}
\end{defn}
We say that the family $\{S(t,s)\}_{(t,s)\in\mathfrak{T}}$
\emph{solves} non-autonomous Cauchy problem (\ref{NACP1}) if there
exist a family $(Y_s)_{s\in[0,T]}$ of dense subspaces of $\bE$
such that for all $(s,t)\in\mathfrak{T}$ we have
\[
S(t,s)Y_s\subset Y_t\subset D(A(t))\]
and the map
$y(t)=S(t,s)x$ is a strict solution of (\ref{NACP1}) for every
$x\in Y_s.$ In this case we say that $\set{-A(t),D(A(t))}_{t\in
[0,T]}$ (or simply $\set{-A(t)}_{t\in [0,T]})$ generates the evolution family
$\{S(t,s)\}_{(t,s)\in\mathfrak{T}}.$
%We start off by recalling some results on sectorial operators.

Under the condition \textbf{(AT)} we have the following well-known result (see Acquistapace and Terreni \cite[Theorems 6.1-6.4]{AcquiTer87} and Yagi \cite[Theorem 2.1]{yagi}).
\begin{thm}
If condition {\rm\textbf{(AT)}} holds then there exists a unique strongly continuous evolution family $\{S({t,s})\}_{(t,s)\in\mathfrak{T}}$ that solves the non-autonomous Cauchy problem $(\ref{NACP1})$ with $Y_t=D(A(t))$ and for all $x\in\bE,$ the map $y(t)=S(t,s)x$
is a classical solution of $(\ref{NACP1}).$ Moreover, $\{S({t,s})\}_{(t,s)\in\mathfrak{T}}$ is continuous on $0\le s<t\le T$ and there exists a constant $C>0$ such that for all $0\le s<t\le T$ and $\theta\in[0,1],$
\begin{align*}
\bigl|\!\bigl|(A(t)+wI)^\theta S(t,s)\bigr|\!\bigr|_{\Lin(\bE)}&\le C(t-s)^{-\theta}\\
\bigl|\!\bigr|S(t,s)-e^{-(t-s)A(s)}\bigr|\!\bigr|_{\Lin(\bE)}&\le C(t-s)^{\mu+\nu-1}
\end{align*}
Moreover, for all $\theta\in (0,\mu)$ and $x\in D\left((A(t)+w I)^\theta\right)$ we have
\begin{equation}\label{mutheta}
\abs{S(t,s)(A(t)+w I)^\theta x}_{\bE}\le C(\mu-\theta)^{-1}(t-s)^{-\theta}\abs{x}_\bE.
\end{equation}
\end{thm}

\section{Backward Ornstein-Uhlenbeck transition evolution\\operators}
Let $\set{-A(t)}_{t\in [0,T]}$ be the generator of an evolution family $\{S(t,s)\}_{(t,s)\in\mathfrak{T}}$ on $\bE$ and let $\set{G(t)}_{t\in [0,T]}$ be closed operators from a constant domain $D(G)\subset\bH$ into $\bE.$ We start this section by discussing the existence of mild solutions to the non-autonomous linear stochastic equation
\begin{equation}\label{linz}
\begin{split}
dZ(t)+A(t)Z(t)\,dt&=G(t)\,dW(t), \ \ \ \ t\in [s,T],\\
Z(s)&=x_0\in\bE,
\end{split}
\end{equation}
with moving time origin $s\in[0,T]$ and initial data $x_0\in\bE.$

\begin{defn}
We say that an $\bE-$valued process $Z(\cdot)$ is a \emph{mild solution} of (\ref{linz}) if for all $(t,s)\in\mathfrak{T}$ the mapping $S(t,s) G(s)$ has a continuous extension to a bounded operator from $\bH$ into $\bE,$ which we will also denote by $S(t,s) G(s),$ such that the operator-valued function $(s,t)\ni r \mapsto S(t,r)G(r)\in\Lin(\bH,\bE)$ is stochastically integrable on the interval $(s,t)$ and
\[
Z(t)=S(t,s)x_0+\int_s^t S(t,r)G(r)\,dW(r), \ \ \ \Prob-\text{a.s}
\]
\end{defn}
%In what follows, $\bE^*$ will denote the dual of $\bE$ and $\seq{\cdot,\cdot}$ will denote the duality pairing between $\bE$ and $\bE^*.$
We know from Theorem \ref{vnw1thm4.2} that existence of a mild solution for (\ref{linz}) follows from the following condition
\begin{Assumption}\label{assu1}
For each $(t,s)\in\mathfrak{T}$ the mapping $S(t,s)G(s):D(G)\to\bE$ extends to a bounded linear operator from $\bH$ into $\bE,$ also denoted by $S(t,s) G(s),$ such that the positive symmetric operator $Q_{t,s}\in\Lin(\bE^*,\bE)$ defined by
\begin{equation}\label{Qtsdef}
\seq{Q_{t,s}x^*,y^*}:=\int_s^t\seq{S(t,r)G(r)(S(t,r)G(r))^*x^*,y^*}\,dr, \ \ \ x^*,y^*\in\bE^*
\end{equation}
is the covariance operator of a centered Gaussian measure $\mu_{t,s}$ on $\bE.$
\end{Assumption}

\noindent\textbf{Notation.} For each $(t,s)\in\mathfrak{T},$ let $(H_{t,s,}[\cdot,\cdot]_{H_{t,s}})$ denote the Reproducing Kernel Hilbert Space associated with the positive symmetric operator $Q_{t,s}$ defined by (\ref{Qtsdef}), and let $i_{t,s}$ denote the inclusion mapping from $H_{t,s}$ into $\bE.$

\begin{ex}\label{ex1}
Let $\bE$ be a type-2 Banach space and suppose that for each $(t,s)\in\mathfrak{T}$ we have $S(t,s)G(s)\in\gamma(\bH,\bE)$ and
\begin{equation}\label{stg2gamma}
\int_0^T|\!|S(t,s) G(s)|\!|^2_{\gamma(\bH,\bE)}\,dt<+\infty.
\end{equation}
Then, by Theorem \ref{vnw1thm4.7}, Assumption \ref{assu1} holds.
\end{ex}

For the next example, consider the following linear parabolic second-order stochastic PDE perturbed by additive space-time white noise on $[0,T]\times (0,1),$
\begin{equation}\label{exlinearspde}
\begin{split}
\frac{\partial X}{\partial t}(t,\xi)+(\mathcal{A}_t X)(t,\xi)&=g(t,\xi)\,\frac{\partial w}{\partial t}(t,\xi),\\
X(t,0)=X(t,1)&=0,\\
X(0,\cdot)&=x_0(\cdot),
\end{split}
\end{equation}
where, for each $t\in [0,T],$ $\A_t$ is the second-order differential operator
\[
(\mathcal{A}_t x)(\xi):=-a(t,\xi)\frac{d^2x}{d\xi^2}(\xi)+b(t,\xi)\frac{d x}{d\xi}(\xi)+c(t,\xi)x(\xi), \ \ \ \xi\in (0,1)
\]
with $a,b,c\in\mathcal{C}^\mu([0,T];\mathcal{C}([0,1]))$ and $a\in\mathcal{C}^\varepsilon([0,1];\mathcal{C}([0,T]))$ for $\mu\in (\frac{1}{4},1]$  and $\varepsilon>0$ fixed. Assume further that $\inf_{t\in [0,T], \xi\in [0,1]}a(t,\xi)>0.$

For $p\geq 2$ and $t\in [0,T],$ let $A_p(t)$ denote the realization in $L^p(0,1)$ of $\mathcal{A}_t$ with zero-Dirichlet boundary conditions,
\begin{align*}
    D(A_p(t))&:=H^{2,p}(0,1)\cap H_0^{1,p}(0,1),\\
    A_p(t)&:=\mathcal{A}_t.
\end{align*}
It is well-known that for $w$ sufficiently large, the operator $A_p(\cdot)+w I$ satisfies \textbf{(AT)} with parameters $\mu$ and $\nu=1$ (see e.g. Acquistapace and Terreni \cite{AcquiTer87} or Tanabe \cite{tanabe}). We will assume for simplicity and without loss of generality that $w=0.$

Let $\set{S_p(t,s)}_{(t,s)\in\mathfrak{T}}$ denote the family of evolution operators generated by $\set{-A_p(t)}_{t\in [0,T]}.$ Let $g\in L^1(0,T;L^\infty(0,1))$ be fixed and define, for almost every $t\in [0,T],$ the multiplication operators from $L^2(0,1)$ into $L^p(0,1)$ as follows
\begin{align*}
    D(G(t))&:=L^p(0,1)\subset L^2(0,1)\\
    G(t)y&:=\set{(0,1)\ni\xi\mapsto g(t,\xi)y(\xi)\in\R}.
\end{align*}
Notice that $G(t)$ is not a bounded operator unless $p=2.$ However, we can prove the following
\begin{ex}\label{ex2}
For each $(t,s)\in\mathfrak{T}$ the map $S_p(t,s)G(s)$ can be extended to bounded operator from $L^2(0,1)$ into $L^p(0,1)$ that is $\gamma-$radonifying and satisfies (\ref{stg2gamma}). Since $L^p(0,1)$ has type-2 for $p\geq 2,$ from Example \ref{ex1} it follows that Assumption \textbf{\ref{assu1}} holds for $\set{S_p(t,s)G(s)}_{(t,s)\in\mathfrak{T}}$ with $\bH=L^2(0,1)$ and $\bE=L^p(0,1)$ with $p\geq 2.$ Equivalently, equation (\ref{exlinearspde}) has a mild solution in $L^p(0,1).$
\begin{proof}
The argument of the proof follows closely ideas from Veraar and Zimmerschied \cite[Section 5]{veraarzimmer1}. We show first that if $\sigma>\frac{1}{4}$ then $A_p(t)^{-\sigma}$ extends to a bounded operator from $L^2(0,1)$ into $L^p(0,1)$, which we also denote by $A_p(t)^{-\sigma},$ such that
\begin{equation}\label{Apgamma}
A_p(t)^{-\sigma}\in\gamma\left(L^2(0,1),L^p(0,1)\right).
\end{equation}
We know from Example \ref{ex0} that the identity operator on $D(\Delta_2)$ extends to a continuous embedding $j:D(\Delta_2)\embed D(\Delta_p^{1-\sigma})$ which is $\gamma-$radonifying. Moreover, the family of operators
\[
\set{A_p(t)A_p(s)^{-1}:s,t,\in[0,T]}
\]
is uniformly bounded in $\Lin(L^p(0,1))$ (see e.g. Tanabe \cite[Section 5.2]{tanabe}). This implies, in particular, that both domains $D(A_p(t))$ and $D(A_p(0))$ coincide with equivalent norms, uniformly in $t\in [0,T].$ Since $D(A_p(0))=D(\Delta_p)$ with equivalent norms, we conclude that $D(A_p(t))=D(\Delta_p)$ with equivalent norms uniformly in $t\in [0,T].$

Using the $\varepsilon-$H\"{o}lder continuity assumption on the coefficients of $\A_t,$ it can be proved that the operators $A_p(t)$ belong to $\BIP(L^p(0,1),\phi)$ for some $\phi>0,$ see e.g. Denk et al \cite{denketal} or Pr\"{u}ss and Sohr \cite{pruesohr2}. Hence, by Theorem 1.15.3 in Triebel \cite{triebel} we have
\[
D(A_p(t)^{1-\sigma})=[L^p(0,1),D(A_p(t))]_{1-\sigma}=[L^p(0,1),D(\Delta_p)]_{1-\sigma}=D(\Delta_p^{1-\sigma})
\]
isomorphically, with equivalence in norm uniformly in $t\in [0,T].$ Therefore, by the ideal property of $\gamma\left(D(\Delta_2), D(\Delta_p^{1-\sigma})\right),$ we obtain
\[
A_p(t)^{-\sigma}=A_p(t)^{1-\sigma}jA_2(t)^{-1}\in \gamma\left(L^2(0,1),L^p(0,1)\right)
\]
with $\norm{A_p(t)^{-\sigma}}_{\gamma\left(L^2(0,1),L^p(0,1)\right)}$ uniformly bounded in $t\in [0,T].$

Now, from (\ref{mutheta}) it follows that, if $\sigma\in(0,\mu)$ then the map $S_p(t,s)A_p(s)^\sigma$ extends to a bounded operator $S_{p,\sigma}(t,s)$ on $L^p(0,1)$ with
\[
\norm{S_{p,\sigma}(t,s)}_{\Lin(L^p(0,1))}\le C(\mu-\sigma)^{-1}(t-s)^{-\sigma}.
\]
Hence, again by the ideal property of $\gamma-$radonifying operators, we conclude that if $\sigma\in (\frac{1}{4},\mu),$ for each $(t,s)\in\mathfrak{T}$ the linear mappings
\[
S_p(t,s)G(s)=S_p(t,s)A_p(s)^\sigma A_p(s)^{-\sigma}G(s)
 \] have a continuous extension to bounded operators from $L^2(0,1)$ into $L^p(0,1)$ that are $\gamma-$radonifying and satisfy (\ref{stg2gamma}).
\end{proof}
\end{ex}
We now introduce the transition evolution operators associated with the linearized equation (\ref{linz}). Suppose that Assumptions \textbf{(AT)} and \textbf{\ref{assu1}} are satisfied. Let $\B_b(\bE)$ denote the set of Borel-measurable bounded real-valued functions on $\bE.$
\begin{defn}\label{ou}
The \emph{Ornstein-Uhlenbeck (OU) transition evolution operators} $\{P(s,t)\}_{(t,s)\in\mathfrak{T}}$
associated to equation (\ref{linz}) are defined by
\[
[P(s,t)\varphi](x):=\int_\bE\varphi(S(t,s)x+z)\,\mu_{t,s}(dz), \ \ x\in\bE, \ \varphi\in\B_b(\bE), \ \ (t,s)\in\mathfrak{T}
\]
\end{defn}
Before we discuss the smoothing property of the OU transition operators, we extend to the non-autonomous framework some results by van Neerven \cite[Section 1]{vn1} on the relation between the spaces $H_{t,s}$ for different values of $s<t.$ The first observation is the following algebraic relation between the operators $Q_{t,s},$ which is immediate from their definition
\[
Q_{t,s}=Q_{t,r}+S(t,r)Q_{r,s}S(t,r)^*, \ \ \ \ 0\le s<r<t.
\]
The following is a direct consequence of Proposition \ref{vn1Prop1.1},
\begin{prop}\label{vn1Prop1.3}
$H_{t,r}\subset H_{t,s}$ for all $0\le s<r<t.$
\end{prop}
This combined with the identity $S(t,r)Q_{r,s}S(t,r)^*=Q_{t,s}-Q_{t,r},$ implies that $S(t,r)$ maps the linear subspace $\range Q_{r,s}S(t,r)^*$ of $H_{r,s}$ into $H_{t,s}.$ The next result shows that we actually have $S(t,r)H_{r,s}\subset H_{t,s}.$

\begin{thm}\label{vn1thm1.4}
For all $0\le s<r<t$ we have $S(t,r)H_{r,s}\subset H_{t,s}.$ Moreover
\[
\norm{S(t,r)}_{\Lin(H_{r,s},H_{t,s})}\le 1.
\]
\end{thm}
\begin{proof}
For all $x^*\in\bE^*$ we have
\begin{align}
\abs{Q_{r,s}S(t,r)^*x^*}_{H_{r,s}}^2&=\seq{Q_{r,s}S(t,r)^*x^*,S(t,r)^*x^*}\notag\\
&=\seq{Q_{t,s}x^*,x^*}-\seq{Q_{t,r}x^*,x^*}\label{vn1-1.1}\\
&\le \seq{Q_{t,s}x^*,x^*}=\abs{Q_{t,s}x^*}_{H_{t,s}}^2.\notag
\end{align}
Hence,
\begin{align}
\abs{\seq{Q_{r,s}S(t,r)^*x^*,y^*}}&=\abs{[Q_{r,s}S(t,r)^*x^*,Q_{r,s}y^*]_{H_{r,s}}}\notag\\
&\le \abs{Q_{t,s}x^*}_{H_{t,s}}\abs{Q_{r,s}y^*}_{H_{r,s}}.\label{vn1-1.2}
\end{align}
For $y^*\in\bE^*$ fixed  we define the linear functional $\psi_{y^*}:\range Q_{t,s}\to\R$ by
\[
\psi_{y^*}(Q_{t,s}x^*):=\seq{Q_{r,s}S(t,r)^*x^*,y^*}.
\]
This is well-defined since, by (\ref{vn1-1.1}), if $Q_{t,s}x^*=0$ then $Q_{r,s}S(t,r)^*x^*=0.$ By (\ref{vn1-1.2}) $\psi_{y^*}$ extends to a bounded linear functional on $H_{t,s}$ with norm bounded by $\abs{Q_{r,s}y^*}_{H_{r,s}}.$ Identifying $\psi_{y^*}$ with an element of $H_{t,s},$ for all $x\in\bE^*$ we have
\[
\seq{\psi_{y^*},x^*}=[Q_{t,s}x^*,\psi_{y^*}]_{H_{t,s}}=\seq{Q_{r,s}S(t,r)^*x^*,y^*}=\seq{S(t,r)Q_{r,s}y^*,x^*}.
\]
Therefore, $S(t,r)Q_{r,s}y^*=\psi_{y^*}\in H_{t,s}$ and $\abs{S(t,r)Q_{r,s}y^*}_{H_{t,s}}\le \abs{Q_{r,s}y^*}_{H_{r,s}},$ and the desired result follows.
\end{proof}

Next we characterize the equality of the Hilbert spaces $H_{t,r}$ and $H_{t,s}$ in terms of the restriction $S(t,r)\in\Lin(H_{r,s},H_{t,s}).$
\begin{thm}\label{vn1thm1.7}
For all $0\le s<r<t$ we have $H_{t,s}=H_{t,r},$ as subsets of $\bE,$ if and only if $\norm{S(t,r)}_{\Lin(H_{r,s},H_{t,s})}<1.$
\end{thm}
\begin{proof}
We know already that $H_{t,r}\subset H_{t,s},$ so it remains to prove that $H_{t,s}\subset H_{t,r},$ if and only if $\norm{S(t,r)}_{\Lin(H_{r,s},H_{t,s})}<1.$

We assume first that $\norm{S(t,r)}_{\Lin(H_{r,s},H_{t,s})}<1.$ By Theorem \ref{vn1thm1.4}, for $y^*\in\bE^*$ we have $S(t,r)Q_{r,s}y^*\in H_{t,s}.$ Then, if $x^*\in\bE^*$ it follows that
\begin{equation}\label{QrsStrHrs}
\begin{split}
[Q_{r,s}S(t,r)^*x^*,Q_{r,s}y^*]_{H_{r,s}}&=\seq{S(t,r)^*x^*,Q_{r,s}y^*}\\
&=\seq{x^*,S(t,r)Q_{r,s}y^*}\\
&=[Q_{t,s}x^*,S(t,r)Q_{r,s}y^*]_{H_{t,s}}.
\end{split}
\end{equation}
Hence
\begin{align*}
&\abs{Q_{r,s}S(t,r)^*x^*}_{H_{r,s}}\\
&\phantom{mm}=\sup\set{[Q_{r,s}S(t,r)^*x^*,Q_{r,s}y^*]_{H_{r,s}}:y^*\in\bE^*, \ \abs{Q_{r,s}y^*}_{H_{r,s}}\le 1}\\
&\phantom{mm}=\sup\set{[Q_{t,s}x^*,S(t,r)Q_{r,s}y^*]_{H_{t,s}}:y^*\in\bE^*, \ \abs{Q_{r,s}y^*}_{H_{r,s}}\le 1}\\
&\phantom{mm}\le \norm{S(t,r)}_{\Lin(H_{r,s},H_{t,s})}\cdot \abs{Q_{t,s}x^*}_{H_{t,s}}.
\end{align*}
Using the last inequality, we get
\begin{align*}
  \abs{Q_{t,s}x^*}_{H_{t,s}}^2&=\abs{Q_{t,s}x^*}_{H_{t,s}}^2-\abs{Q_{r,s}S(t,r)^*x^*}^2_{H_{r,s}}+\abs{Q_{r,s}S(t,r)^*x^*}^2_{H_{r,s}}\\
  &= \seq{Q_{t,s}x^*,x^*}-\seq{S(t,r)Q_{r,s}S(t,r)^*x^*,x^*}+\abs{Q_{r,s}S(t,r)^*x^*}^2_{H_{r,s}}\\
  &\le \seq{Q_{t,r}x^*,x^*}+\norm{S(t,r)}_{\Lin(H_{r,s},H_{t,s})}^2\cdot \abs{Q_{t,s}x^*}_{H_{t,s}}^2
\end{align*}
That is,
\[
\seq{Q_{t,s}x^*,x^*}=\abs{Q_{t,s}x^*}_{H_{t,s}}^2\le\frac{1}{1-\norm{S(t,r)}_{\Lin(H_{r,s},H_{t,s})}^2}\seq{Q_{t,r}x^*,x^*}.
\]
By Proposition \ref{vn1Prop1.1}, this implies the inclusion $H_{t,s}\subset H_{t,r}.$ Conversely, assume that $H_{t,s}\subset H_{t,r}.$ Then there exists $K>0$ such that
\begin{align*}
\seq{Q_{t,s}x^*,x^*}&\le K\seq{Q_{t,r}x^*,x^*}\\
&=K\seq{Q_{t,s}x^*.x^*}-K\seq{S(t,r)Q_{r,s}S(t,r)^*x,x^*}
\end{align*}
for all $x^*\in\bE^*.$ Notice that $K>1$ since $\seq{Q_{t,r}x^*,x^*}\le \seq{Q_{t,s}x^*,x^*}$ for all $x\in\bE^*.$ Then, the above inequality yields
\[
\abs{Q_{r,s}S(t,r)^*x}_{H_{r,s}}^2\le\textstyle{\left(1-\frac{1}{K}\right)}\abs{Q_{t,s}x^*}_{H_{t,s}}^2.
\]
Using (\ref{QrsStrHrs}) again we get
\begin{align*}
\abs{[S(t,r)Q_{r,s}y^*,Q_{t,s}x^*]_{H_{t,s}}} &=\abs{[Q_{r,s}y^*,Q_{r,s}S(t,r)^*x^*]_{H_{r,s}}}\\
&\le \abs{Q_{r,s}y^*}_{H_{r,s}}\cdot\abs{Q_{r,s}S(t,r)^*x^*}_{H_{r,s}}\\
&\le \textstyle{\left(1-\frac{1}{K}\right)^{1/2}}\abs{Q_{r,s}y^*}_{H_{r,s}}\cdot\abs{Q_{t,s}x^*}_{H_{t,s}}
\end{align*}
which shows that $\norm{S(t,r)}_{\Lin(H_{r,s},H_{t,s})}\le \left(1-\frac{1}{K}\right)^{1/2}<1.$
\end{proof}
Finally, we establish the smoothing property of the Ornstein-Uhlenbeck transition operators. We need the following assumption, usually referred to as null-controllability condition (see Remark \ref{nullcontrol} below).
\begin{Assumption}\label{assu2}
For all $(t,s)\in\mathfrak{T}$ we have
\begin{equation}\label{eq6.2.3}
\range S(t,s)\subset H_{t,s}
\end{equation}
\end{Assumption}

\noindent\textbf{Notation.} If condition (\ref{eq6.2.3}) holds, we denote by $\Sigma(t,s)$ the map $S(t,s)$ regarded as an operator from $\bE$ into $H_{t,s}.$ Notice that $\Sigma(t,s)$ is bounded by the Closed-Graph Theorem, and we have $S(t,s)=i_{t,s}\circ \Sigma(t,s).$

As in Definition \ref{phimu}, let $\phi_{t,s}:H_{t,s}\to L^2(\bE,\mu_{t,s})$ denote the unique bounded extension of the isometry
\[
Q_{t,s}(\bE^*)\ni Q_{t,s}x^*\mapsto \seq{x^*,\cdot}\in L^2(\bE,\mu_{t,s}).
\]
Let $C_b^\infty(\bE)$ denote the set of infinitely Fr\'{e}chet-differentiable real-valued functions on $\bE.$ Using Proposition \ref{sp0} together with the condition (\ref{eq6.2.3}) we obtain the following
\begin{thm}\label{6.2.2}
Let Assumptions {\rm\textbf{(AT)}}, {\rm\textbf{\ref{assu1}}} and {\rm\textbf{\ref{assu2}}} be satisfied. Then the
Ornstein-Uhlenbeck transition operators $\set{P(s,t)}_{(t,s)\in\mathfrak{T}}$ satisfy
\[
\varphi\in\B_b(\bE)\Rightarrow P(s,t)\varphi\in \mathcal{C}_b^\infty(\bE).
\]
The Fr\'{e}chet derivative of the function
$P(s,t)\varphi:\bE\to\R$ at $x\in\bE$ in the direction $y\in\bE$ is
given by
\[
\seq{DP(s,t)\varphi(x),y}=\int_\bE \varphi(S(t,s)
x+z)\,\phi_{t,s}(\Sigma(t,s)y)(z)\,\mu_{t,s}(dz),
\]
and the second Fr\'{e}chet derivative of $P(s,t)\varphi$ at $x\in\bE$
in the directions $y_1,y_2\in\bE$ is given by
\begin{align*}
&\seq{D^2 P(s,t)\varphi(x)y_1,y_2}=-P(s,t)\varphi(x)\left[\Sigma(t,s)y_1,\Sigma(t,s)y_2\right]_{H_{t,s}}\\
&\phantom{mmmmmm}+\int_\bE \varphi(S(t,s)x+z)\,\phi_{t,s}(\Sigma(t,s)y_1)(z)\,\phi_{t,s}(\Sigma(t,s)y_2)(z)\,\mu_{t,s}(dz)
\end{align*}
In particular, we have the estimates
\begin{align}
\norm{D_x
P(s,t)\varphi(x)}_{\bE^*}&\le\norm{\Sigma(t,s)}_{\Lin(\bE,H_{t,s})}\abs{\varphi}_0\label{eq3a}\\
\norm{D^2_x P(s,t)\varphi(x)}_{\Lin(\bE,\bE^*)}&\le
2\norm{\Sigma(t,s)}^2_{\Lin(\bE,H_{t,s})}\abs{\varphi}_0\label{eq3b}.
\end{align}
\end{thm}

\begin{rem}\label{nullcontrol}
The condition (\ref{eq6.2.3}) has a well-known control theoretic
interpretation: for each $s\in [0,T]$ consider the nonhomogeneous Cauchy problem
\begin{equation}\label{ccp}
\begin{split}
    y'(t)+A(t)y(t)&=G(t)u(t), \ \ \ \ t\in[s,T],\\
    y(s)&=x\in\bE,
\end{split}
\end{equation}
with $u\in L^2(s,T;\bH).$ The \emph{mild solution} of (\ref{ccp}) is defined as
\begin{equation}
y^{x,u}(t):=S(t,s) x+\int_s^t S(t,r)G(r)u(r)\,dr, \ \ \ t\in[s,T].
\end{equation}
We say that (\ref{ccp}) is \textit{null-controllable} in time $t$ iff for all $x\in\bE$ there exists a \textit{control} $u\in L^2(s,t
;\bH)$ such that $y^{x,u}(t)=0.$ Using the following characterization of the Hilbert spaces $H_{t,s},$
\begin{equation}\label{crkhst}
H_{t,s}=\left\{\int_t^s S(t,r)G(r)u(r)\,dr:u\in L^2(s,t;\bH)\right\}, \ \ \
(t,s)\in\mathfrak{T}
\end{equation}
(see e.g. van Neerven \cite[Lemma 5.2]{vn2}) it follows that (\ref{ccp}) is null-controllable in time $t$ if and
only if condition (\ref{eq6.2.3}) holds, and we have
\begin{equation}\label{normht}
\abs{x}_{H_{t,s}}=\inf\set{\abs{u}_{L^2(s,t;\bH)}:u\in L^2(s,t;\bH) \ \text{ and } \ \int_s^t S(t,r)G(r)u(r)\,dr=x}.
\end{equation}
That is, $\abs{x}_{H_{t,s}}^2$ is the minimal energy needed to steer the control system (\ref{ccp}) from $0$ to $x$ in time $t-s.$
%$H_t$ coincides with the subset of $\bE$ of states reachable in time $t$ by the solution of (\ref{ccp}) starting at $x=0.$
\end{rem}

\begin{ex}\label{ex3}
Suppose that for each $t\in[0,T]$ the map $G(t)$ is injective and for each $(t,s)\in\mathfrak{T}$ we have
\[
\range S(t,s)\subset \range G(t).
\]
Suppose also that for each $s\in [0,T]$ we have
\[
\int_s^T|\!|G(t)^{-1}S(t,s)|\!|^2_{\Lin(\bE,\bH)}\,dt<+\infty.
\]
Then Assumption \ref{assu2} holds. Indeed, let $x\in\bE$ and $0\le s<t\le T,$ and define
%$G$ is boundedly invertible, that is, $G^{-1}$ exists and is a bounded operator from $\bE$ into $\bH.$
\[
u(r):=\frac{1}{t-s}G(r)^{-1}S(r,s)x, \ \ \ r\in [s,t].
\]
Then $u\in L^2(s,t;\bH)$ and we have
\[
\int_s^t S(t,r)G(r)u(r)\,dr=\frac{1}{t-s}\int_s^t S(t,r)S(r,s)x\,dr=S(t,s)x
\]
that is, $S(t,s)x\in H_{t,s}$ according to (\ref{crkhst}), and Assumption \ref{assu2} follows. Moreover, by (\ref{normht}), we have
\[
\abs{\Sigma(t,s)x}_{H_{t,s}}\le \frac{1}{t-s}\abs{x}_{\bE}\left(\int_s^t|\!|G(r)^{-1}S(r,s)|\!|^2_{\Lin(\bE,\bH)}\,dr\right)^{1/2}, \ \ \ s<t\le T.
\]
\end{ex}

\section{Mild solutions of Hamilton-Jacobi equations in Banach spaces}
%variation of constants formula, integral form of the HJB equation
Let $\set{-A(t)}_{t\in [0,T]}$ be the generator of an evolution family on a Banach space $\bE.$ Let $\bH$ be a separable Hilbert space and let $\set{G(t)}_{t\in [0,T]}$ be a family of (possibly unbounded) linear operators from $\bH$ into $\bE.$ We consider the Hamilton-Jacobi equation on $\bE$
\begin{equation}\label{hjb1}
\begin{split}
\frac{\partial v}{\partial t}(t,x)+L_tv(t,\cdot)(x)+\cH(t,x,D_xv(t,x))&=0, \ \ \ (t,x)\in [0,T]\times\bE,\\
v(T,x)&=\varphi(x).
\end{split}
\end{equation}
The final condition $\varphi:\bE\to\R$ and the nonlinear \emph{Hamiltonian} operator $\cH:[0,T]\times\bE\times\bE^*\to\R$ are given, and for each $t\in[0,T],$ $L_t$ is the second-order differential operator
\[
(L_t\phi)(x):=-\seq{A(t)x,D_x \phi(x)}+\frac{1}{2}\Tr_\bH[G(t)^*D_x^2\phi(x)G(t)],
\]
for $x\in D(A(t)), \ \phi\in \mathcal{C}_b^2(\bE).$ Using the associated OU-transition evolution operators $\{P(s,t)\}_{(t,s)\in\mathfrak{T}}$ (see Definition \ref{ou}) we rewrite equation (\ref{hjb1}) in the integral form
\begin{equation}\label{hjb2}
v(t,x)=[P(t,T)\varphi](x)+\int_t^T \left[P(t,s)\cH(s,\cdot,D_x v(s,\cdot))\right](x)\,ds.
\end{equation}
Observe that the trace term in (\ref{hjb1}) may not be well-defined since $G(t)$ is not necessarily a bounded operator.
\begin{defn}
For $\alpha\in (0,1),$ we denote with $\mathfrak{S}_{T,\alpha}$ the set of bounded and measurable functions $v:[0,T]\times\bE\to\R$ such that
$v(t,\cdot)\in \mathcal{C}_b^1(\bE),$ for all $t\in [0,T),$ and the mapping
\[
[0,T)\times\bE\ni(t,x)\mapsto (T-t)^\alpha D_x v(t,x)\in\bE^*
\]
is bounded and measurable.
\end{defn}
The space $\mathfrak{S}_{T,\alpha}$ is a Banach space endowed with the norm
\[
\norm{v}_{\mathfrak{S}_{T,\alpha}}:=\sup_{t\in [0,T]}\norm{v(t,\cdot)}_{0}+\sup_{t\in [0,T]} \ (T-t)^\alpha \norm{D_xv(t,\cdot)}_{0}.
\]
\begin{defn}
We will say that a function $v:[0,T]\times\bE\to\R$ is a \emph{mild} solution of the Hamilton-Jacobi equation (\ref{hjb1}) if $v\in\mathfrak{S}_{T,\alpha}$ for some $\alpha\in (0,1),$ for each $(t,x)\in [0,T]\times\bE$ the mapping
\[
[t,T]\ni s\mapsto [P(t,s)\cH(s,\cdot,D_x v(s,\cdot))](x)\in\R
\]
is integrable and $v$ satisfies (\ref{hjb2}).
\end{defn}

\begin{Assumption}\label{assu3}
There exists $\alpha\in(0,1)$ and $C>0$ such that
\[
|\!|\Sigma(t,s)|\!|_{\Lin(\bE,H_{t,s})}\le C{(t-s)^{-\alpha}}, \ \ \ \ 0\le s <t\le T.
\]
\end{Assumption}
\begin{ex}\label{ex4}
Under the same assumptions of Example \ref{ex3}, assume further that there exists $\beta\in[0,\frac{1}{2})$ and $C>0$ such that
\[
|\!|G(t)^{-1}S(t,s)|\!|_{\Lin(\bE,\bH)}\le C(t-s)^{-\beta}, \ \ \ 0\le s<t\le T.
\]
Then Assumption \ref{assu3} holds with $\alpha=\beta+\frac{1}{2}.$
\end{ex}
\begin{Assumption}\label{assu4}
For all $(t,p)\in [0,T]\times\bE^*,$ the map
\[
\bE\ni x\mapsto\cH(t,x,p)\in\R
\]
is continuous and bounded, and there exists $C>0$ such that
\[
\abs{\cH(t,x,p)-\cH(t,x,q)}\le C\abs{p-q}_{\bE^*}, \ \ \ \ t\in [0,T], \ \ x\in\bE, \ \ p,q\in\bE^*.
\]
\end{Assumption}

\begin{ex}
If the Hamiltonian $\cH$ has the form
\begin{equation}\label{ham0}
\cH(t,x,p)=\inf_{u\in M}\set{\seq{F(t,x,u),p}+h(t,x,u)}, \ (t,x,p)\in [0,T]\times\bE\times\bE^*
\end{equation}
where $M$ is a separable metric space, $F:[0,T]\times\bE\times M\to\bE$ is uniformly bounded and weakly-continuous in $x\in\bE$ uniformly with respect to $u\in M,$ and $h:[0,T]\times \bE\times M\to (-\infty,\infty]$ is continuous in $x\in\bE$ uniformly with respect to $u\in M$ and satisfies
\[
\sup_{x\in\bE,u\in M}\abs{h(t,x,u)}<+\infty, \ \forall t\in [0,T]
\]
then Assumption \textbf{\ref{assu4}} holds, see e.g. the proof of Theorem 10.1 in Fleming and Soner \cite[Chapter II]{fs1} or Bardi and Capuzzo-Dolcetta \cite[Chapter III, Lemma 2.11]{bardi}.

Recall that if Hamiltonian $\cH$ takes the form (\ref{ham0}), equation (\ref{hjb1}) is the \emph{Hamilton-Jacobi-Bellman} PDE associated with the dynamic programming approach to stochastic optimal control problems of the form
\begin{equation}
\text{minimize}  \ \ \ J(X,u)=\Exp\left[\int_0^T h(t,X(t),u(t))\,dt+\varphi(X(T))\right]
\end{equation}
subject to
\begin{itemize}
  \item $\{u(t)\}_{t\in [0,T]}$ is an $M$-valued control process
  \item $\set{X(t)}_{t\in [0,T]}$ is the $\bE$-valued solution to the controlled non-autonomous stochastic evolution equation with additive noise
\begin{equation*}
\begin{split}
dX(t)+A(t)X(t)\,dt&=F(t,X(t),u(t))\,dt+G(t)\,dW(t),\\
X(0)&=x_0\in\bE.
\end{split}
\end{equation*}
\end{itemize}
\end{ex}

\begin{rem}
A Banach-space framework seems more suitable for certain control problems for stochastic PDEs. Consider for instance the following controlled stochastic PDE of \emph{reaction-diffusion} type perturbed by additive space-time white noise on
$[0,T]\times (0,1),$
\begin{equation}\label{exspde}
\begin{split}
\frac{\partial X}{\partial t}(t,\xi)+(\mathcal{A}_t X)(t,\xi)&=f(X(t,\xi),u(t))+g(t,\xi)\,\frac{\partial w}{\partial t}(t,\xi),\\
X(t,0)=X(t,1)&=0,\\
X(0,\cdot)&=x_0(\cdot),
\end{split}
\end{equation}
where, for each $t\in [0,T],$ $\mathcal{A}_t$ denotes the second order
differential operator introduced in Example \ref{ex2}. Stochastic PDEs of the form (\ref{exspde}) feature when modeling the concentration, density or temperature of a certain substance under random perturbations. In these applications, it is useful to study running cost functionals that permit to regulate $X(\cdot)$ at some fixed points
$\zeta_1,\ldots,\zeta_n\in(0,1),$ say
\begin{equation}\label{excostfunct}
J(X,u)=\mathds{E}\left[\int_0^T \phi(t,X(t,\zeta_1),\ldots,X(t,\zeta_n),u(t))\,dt\right].
\end{equation}
This running costs clearly requires that the solution $X(t,\xi)$ is continuous with respect to the space variable $\xi.$ Recently, Veraar \cite{veraar1} (see also Veraar and Zimmerschied \cite{veraarzimmer1}) have proved that weak solutions to the uncontrolled version of equation (\ref{exspde}) exist and have trajectories almost surely in
\[
\mathcal{C}([0,T];D(A_p(0)^\delta)) \ \ \ \text{for \ } \delta<\frac{1}{4}
\]
where $A_p(t)$ denotes the realization in $L^p(0,1)$ of $\A_t$ with zero-Dirichlet boundary conditions, see Example \ref{ex2}. If we choose
\[
p>2 \ \ \text{ and  } \ \ \delta\in\left(\frac{1}{2p},\frac{1}{4}\right)
\]
using Theorem 1.15.3 in Triebel \cite{triebel} and Sobolev's embedding theorem, it follows
\[
D(A_p(0)^\delta)=[L^p(0,1),D(A_p(0))]_\delta=H^{2\delta,p}_0(0,1)\hookrightarrow\mathcal{C}_0[0,1]
\]
that is, cost functional (\ref{excostfunct}) is now well-defined. This suggests to choose $\bE=L^p(0,1)$ with $p>2$ as state space for the above control problem and the associated Hamilton-Jacobi-Bellman equation (\ref{hjb1}).

As we mentioned in the introduction, at the moment we are unable to obtain optimality criteria and verification-type results for optimal control problems in Banach spaces for non-autonomous stochastic PDEs as this requires approximation results of in $\mathcal{C}_b(\bE)$ by smooth functions that do not seem available at the moment in the general Banach-space setting. We will address this issue in a forthcoming paper.
\end{rem}
We now present the final result of this paper, which generalizes to the non-autonomous Banach-space setting Theorem 9.3 in Zabcyck \cite{zabc1} on existence of mild solutions to HJ equations in Hilbert spaces (see also Da Prato and Zabcyck \cite[Part III]{dpz3} and Masiero \cite{masiero1}).
\begin{thm}\label{main}
Let $\varphi\in\mathcal{C}_b(\bE).$ Suppose Assumptions {\rm\textbf{(AT)}} and {\rm\textbf{\ref{assu1}}}-{\rm\textbf{\ref{assu4}}} hold true. Then there exists a unique mild solution to equation $(\ref{hjb1}).$
\end{thm}
\begin{proof}
The argument is largely based on the proof of Theorem 2.9 in Masiero \cite{masiero1}. For any $v\in\mathfrak{S}_{T,\alpha}$ we define the function $\gamma(v)$ by
\[
\gamma(v)(t,x):=[P(t,T)\varphi](x)+\int_t^T \left[P(t,s)\cH(s,\cdot,D_x v(s,\cdot))\right](x)\,ds,
\]
for $(t,x)\in [0,T]\times\bE.$ By Theorem \ref{6.2.2}, estimate (\ref{eq3a}) and Assumptions \ref{assu3}-\ref{assu4}, the map $\gamma(v)$ belongs to $\mathfrak{S}_{T,\alpha}.$ We will show that $\gamma$ is a strict contraction on $\mathfrak{S}_{T,\alpha}$ when endowed with the equivalent norm
\[
\norm{v}_{\beta,\mathfrak{S}_{T,\alpha}}:=\sup_{t\in [0,T]}\exp(-\beta(T-t))\left[\norm{v(t,\cdot)}_{0}+(T-t)^\alpha \norm{D_xv(t,\cdot)}_{0}\right]
\]
and $\beta>0$ is a parameter to be specified below. Let $v_1,v_2\in\mathfrak{S}_{T,\alpha}.$

\noindent\textsc{Step 1.} From Assumptions \ref{assu3}-\ref{assu4} and estimate (\ref{eq3a}), we obtain
\begin{align*}
&\abs{\gamma(v_1)(t,x)-\gamma(v_2)(t,x)}\\
&\phantom{mm}\le\int_t^T\bigl|\left[P(t,s)\left(\cH(s,\cdot,D_x v_1(s,\cdot))-\cH(s,\cdot,D_x v_2(s,\cdot))\right)\right](x)\bigr|\,ds\,\\
&\phantom{mm}\le C\int_t^T\abs{D_xv_1(s,x)-D_xv_2(s,x)}_{\bE^*}\,ds\\
&\phantom{mm}\le C\int_t^T(T-s)^{-\alpha}\exp(\beta(T-s))\norm{v_1-v_2}_{\beta,\mathfrak{S}_{T,\alpha}}\,ds.
\end{align*}
Let $\varepsilon\in (0,1).$ We can estimate the above integral as follows
\begin{align*}
\int_t^T&(T-s)^{-\alpha}\exp(\beta(T-s))\,ds\\
&=(T-t)^{1-\alpha}\int_0^1 r^{-\alpha}\exp(\beta(T-t)r)\,dr\\
&=(T-t)^{1-\alpha}\left[\int_0^\varepsilon r^{-\alpha}\exp(\beta(T-t)r)\,dr+\int_\varepsilon^1 r^{-\alpha}\exp(\beta(T-t)r)\,dr\right]\\
&\le \frac{(T-t)^{1-\alpha}}{1-\alpha}\left[\varepsilon^{1-\alpha}\exp(\beta(T-t)\varepsilon)
+(1-\varepsilon^{1-\alpha})\exp(\beta(T-t))\right]
\end{align*}
Then
\begin{align*}
&\exp(-\beta(T-t))\norm{\gamma(v_1)(t,\cdot)-\gamma(v_2)(t,\cdot)}_0\\
&\phantom{A}\le \frac{C(T-t)^{1-\alpha}}{1-\alpha}\left[\varepsilon^{1-\alpha}\exp(-\beta(T-t)(1-\varepsilon))
+(1-\varepsilon^{1-\alpha})\right]\norm{v_1-v_2}_{\beta,\mathfrak{S}_{T,\alpha}}.
\end{align*}
We may choose $\varepsilon_1\in (0,1)$ such that
\[
\frac{CT^{1-\alpha}}{1-\alpha}(1-\varepsilon_1^{1-\alpha})<\frac{1}{5}.
\]
Now, for $\beta>\frac{1-\alpha}{T(1-\varepsilon_1)}$ the map
\[
[0,T]\ni t\mapsto [\varepsilon_1(T-t)]^{1-\alpha}\exp(-\beta(T-t)(1-\varepsilon_1))
\]
attains its global maximum at $\bar{t}=T-\frac{1-\alpha}{\beta(1-\varepsilon_1)}.$ Hence,
\[
\sup_{t\in[0,T]}[\varepsilon_1(T-t)]^{1-\alpha}\exp(-\beta(T-t)(1-\varepsilon_1))\\
 =\left[\frac{\varepsilon_1(1-\alpha)}{e\beta (1-\varepsilon_1)}\right]^{1-\alpha}
\]
which tends to zero as $\beta\to\infty.$ Therefore, we can choose $\beta_1=\beta_1(\varepsilon_1)$ sufficiently large such that
\[
\sup_{t\in[0,T]}\frac{C[\varepsilon_1(T-t)]^{1-\alpha}}{1-\alpha}\exp(-\beta_1(T-t)(1-\varepsilon_1))<\frac{1}{5}.
\]
Thus, if $\varepsilon\in(\varepsilon_1,1)$ and $\beta>\beta_1(\varepsilon),$ we have
\[
\sup_{t\in[0,T]}\exp(-\beta(T-t))\norm{\gamma(v_1)(t,\cdot)-\gamma(v_2)(t,\cdot)}_0
\le \frac{2}{5}\norm{v_1-v_2}_{\beta,\mathfrak{S}_{T,\alpha}}.
\]
\noindent\textsc{Step 2.} Using again Assumption \ref{assu3} and estimate (\ref{eq3a}), it follows that
\begin{align*}
   &\abs{D_x\gamma(v_1)(t,x)-D_x\gamma(v_2)(t,x)}_{\bE^*}\\
   &\phantom{P(s,t)}\le \int_t^T\abs{D_xP(t,s)\left[\cH(s,\cdot,D_x v_1(s,\cdot))-\cH(s,\cdot,D_x v_2(s,\cdot))\right](x)}_{\bE^*}\,ds\\
   &\phantom{P(s,t)}\le C\int_t^T(s-t)^{-\alpha}\abs{\cH(s,\cdot,D_x v_1(s,\cdot))-\cH(s,\cdot,D_x v_2(s,\cdot))}\,ds\\
   &\phantom{P(s,t)}\le C^2\int_t^T(s-t)^{-\alpha}\abs{D_xv_1(s,x)-D_xv_2(s,x)}_{\bE^*}\,ds\\
   &\phantom{P(s,t)}\le C^2\int_t^T(s-t)^{-\alpha}(T-s)^{-\alpha}\exp(\beta(T-s))\norm{v_1-v_2}_{\beta,\mathfrak{S}_{T,\alpha}}ds.
\end{align*}
For the last integral we have
\begin{align*}
\int_t^T&(s-t)^{-\alpha}(T-s)^{-\alpha}\exp(\beta(T-s))\,ds\\
&=(T-t)^{1-2\alpha}\left[\int_0^\varepsilon (1-r)^{-\alpha}r^{-\alpha}\exp(\beta(T-t)r)\,dr\right.\\
&\phantom{T-t}\left.
+\int_\varepsilon^1 (1-r)^{-\alpha}r^{-\alpha}\exp(\beta(T-t)r)\,dr\right]\\
&\le \frac{(T-t)^{1-2\alpha}}{1-\alpha}\left[(1-\varepsilon)^{-\alpha}\varepsilon^{1-\alpha}\exp(\beta(T-t)\varepsilon)\right.\\
&\phantom{T-t}
+\left.(1-\varepsilon)^{1-\alpha}\varepsilon^{-\alpha}\exp(\beta(T-t))\right].
\end{align*}
Hence,
\begin{align*}
&\exp(-\beta(T-t))(T-t)^\alpha \norm{D_x\gamma(v_1)(t,\cdot)-D_x\gamma(v_2)(t,\cdot)}_0\\
&\phantom{A}\le \frac{C^2(T-t)^{1-\alpha}}{1-\alpha}\left[(1-\varepsilon)^{-\alpha}\varepsilon^{1-\alpha}\exp(\beta(T-t)(\varepsilon-1)\right.\\
&\phantom{AAA}\left.+(1-\varepsilon)^{1-\alpha}\varepsilon^{-\alpha}\right]\norm{v_1-v_2}_{\beta,\mathfrak{S}_{T,\alpha}}.
\end{align*}
As in step 1, we may choose $\varepsilon_2\in(0,1)$ such that
\[
\frac{C_2\varepsilon_2^{-\alpha}[T(1-\varepsilon_2)]^{1-\alpha}}{1-\alpha}<\frac{1}{5}
\]
and $\beta_2=\beta_2(\varepsilon_2)>0$ sufficiently large such that
\[
\sup_{t\in [0,T]}
\frac{C^2[\varepsilon_2(T-t)]^{1-\alpha}(1-\varepsilon_2)^{-\alpha}}{1-\alpha}\exp(\beta_2(T-t)(\varepsilon_2-1)<\frac{1}{5}.
\]
Thus, for $\varepsilon\in(\varepsilon_2,1)$ and $\beta>\beta_2(\varepsilon),$
\begin{align*}
\sup_{t\in[0,T]}&\exp(-\beta(T-t))(T-t)^\alpha \norm{D_x\gamma(v_1)(t,\cdot)-D_x\gamma(v_2)(t,\cdot)}_0\\
&\le \frac{2}{5}\norm{v_1-v_2}_{\beta,\mathfrak{S}_{T,\alpha}}.
\end{align*}
We conclude that for $\varepsilon\in(\max\set{\varepsilon_1,\varepsilon_2},1)$ and $\beta>\max\set{\beta_1(\varepsilon),\beta_2(\varepsilon)}$ we have
\[
\norm{\gamma(v_1)-\gamma(v_2)}_{\beta,\mathfrak{S}_{T,\alpha}}\le \frac{4}{5}\norm{\gamma(v_1)-\gamma(v_2)}_{\beta,\mathfrak{S}_{T,\alpha}}
\]
and the desired result follows from the Banach fixed point Theorem.
\end{proof}

\begin{ex}\label{exHJ:Lp}
Let $p\geq 2$ be fixed and let $A_p(t)$ and $G(t)$ be as in Example \ref{ex2}, with $g$ satisfying $k_1<\abs{g(t,\xi)}<k_2$ for all $(t,\xi)\in [0,T]\times (0,1),$ for some $k_1>0.$ Let $p^*:=p/(p-1)$ and let $\varphi:L^p(0,1)\to\R$ and $\cH:[0,T]\times L^p(0,1)\times L^{p^*}(0,1)\to\R$ satisfy Assumption \textbf{\ref{assu4}}. Consider the following non-autonomous HJ equation on $[0,T]\times L^p(0,1),$
\begin{equation}\label{eqHJ:Lp}
\begin{split}
\frac{\partial v}{\partial t}(t,x)+L_tv(t,\cdot)(x)+\cH(t,x,D_xv(t,x))&=0, \ \ \ (t,x)\in [0,T]\times L^p(0,1),\\
v(T,x)&=\varphi(x),
\end{split}
\end{equation}
where, for each $t\in[0,T],$ $L_t$ denotes the second-order differential operator on $L^p(0,1),$
\[
(L_t\phi)(x):=-\seq{A_p(t)x,D_x \phi(x)}+\frac{1}{2}\Tr_{L^2(0,1)}[G(t)^*D_x^2\phi(x)G(t)],
\]
$x\in D(A_p(t)), \ \phi\in \mathcal{C}_b^2(L^p(0,1)).$ Here $\seq{\cdot,\cdot}$ denotes the duality pairing between $L^p(0,1)$ and $L^{p^*}(0,1).$ As shown in Examples  \ref{ex2}, \ref{ex3} and \ref{ex4}, Assumptions \textbf{\ref{assu1}}-\textbf{\ref{assu3}} also hold. Then, by Theorem \ref{main}, there exists a unique mild solution to HJ equation (\ref{eqHJ:Lp}).
\end{ex}

\section*{Acknowledgement}
The author thanks Zdzis\l aw Brze\'zniak for suggesting using Theorem 3 in Aronszajn \cite[Ch 2, Section 1]{aronszajn} in the proof of Proposition \ref{sp0}.

\bibliography{biblio}
\end{document}